\documentclass[11pt, english, reqno]{amsart}

\usepackage[T1]{fontenc}
\usepackage{babel}
\usepackage{mathrsfs}
\usepackage{amsthm}
\makeindex
\usepackage{multirow}
\usepackage{enumerate}

\usepackage[dvips,top=3.8cm,left=4cm,right=3.8cm,
foot=3.8cm,bottom=4.0cm]{geometry}
\usepackage{amsfonts,amssymb,amsmath,amsthm, booktabs, 
latexsym}
\usepackage{centernot}
\usepackage{tikz-cd}
\usepackage{bbm}
\usepackage[normalem]{ulem}
\usepackage{enumerate}
\usepackage{color}
\usepackage{soul}
\usepackage{caption}
\newtheorem{theorem}{Theorem}[section]
\newtheorem{lemma}[theorem]{Lemma}
\newtheorem{proposition}[theorem]{Proposition}
\newtheorem{corollary}[theorem]{Corollary}
\theoremstyle{definition}
\newtheorem{definition}[theorem]{Definition}
\newtheorem{example}[theorem]{Example}
\theoremstyle{remark}
\newtheorem{remark}[theorem]{Remark}

\numberwithin{equation}{section}

\numberwithin{equation}{section}
\newif\ifPDF
\ifx\pdfoutput\undefined\PDFfalse
\else \ifnum \pdfoutput > 0 \PDFtrue
        \else \PDFfalse
        \fi
\fi

\ifPDF
  
  \usepackage[colorlinks=true,linkcolor=blue, citecolor=blue]{hyperref}%

\else
  \usepackage{color}
  \usepackage[dvips]{graphicx}
  \usepackage[dvips]{hyperref}
\fi

\newcommand{\be}{\begin{equation}}
\newcommand{\ee}{\end{equation}}
\newcommand{\ba}{\begin{aligned}}
\newcommand{\ea}{\end{aligned}}
\newcommand{\A}{{\mathcal{A}}}
\newcommand{\La}{{\mathcal{L}}}
\newcommand{\N}{{\mathbb N}}
\newcommand{\Z}{{\mathbb Z}}

\def\csi1{\circ\sigma^{-1}}

\def\ol{\overline}

\newcommand{\B}{{\mathcal B}}

\begin{document}

\title[Substitution-dynamics and invariant measures for infinite alphabet-path space.]{Substitution-dynamics and invariant measures for infinite alphabet-path space.}

\author{Sergey Bezuglyi}
\address{University of Iowa, Department of Mathematics. Iowa City, IA, 52242, USA.}
\email{sergii-bezuglyi@uiowa.edu}

\author{Palle E. T. Jorgensen}
\address{University of Iowa, Department of Mathematics. Iowa City, IA, 52242, USA.}
\email{palle-jorgensen@uiowa.edu}

\author{Shrey Sanadhya}
\address{Ben Gurion University of the Negev. Department of Mathematics. Be’er Sheva, 8410501, Israel and University of Iowa, Department of Mathematics. Iowa City, IA, 52242, USA.} \email{shrey.sanadhya@mail.huji.ac.il}
\address{Einstein Institute of Mathematics, The Hebrew University of Jerusalem, Edmond
J. Safra Campus, Jerusalem, Jerusalem, 91904, Israel}
\email{shrey.sanadhya@mail.huji.ac.il}



\begin{abstract} We study substitutions on a countably infinite alphabet (without compactification) as Borel dynamical systems. We construct stationary and non-stationary generalized Bratteli-Vershik models for a class of such substitutions, known as \textit{left determined}. In this setting of Borel dynamics, using a stationary generalized Bratteli-Vershik model, we provide a new and canonical construction of shift-invariant measures (both finite and infinite) for the associated class of subshifts.
\end{abstract}

\maketitle

\tableofcontents
\section{Introduction.}\label{sect Intro}

In this paper, we consider a class of substitution dynamical systems on  
a countably  infinite alphabet and discuss the realization of such systems by 
means of generalized Bratteli diagrams. Using the Perron-Frobenius theorem 
for infinite positive matrices, we describe invariant measures of substitution
dynamical systems.  
We present answers to the following questions in our analysis of substitution dynamics over an infinite alphabet (a partial list): Kakutani-Rokhlin towers for substitution dynamical systems on infinite alphabets, stationary and non-stationary generalized Bratteli-Vershik models for substitution dynamical systems on infinite alphabets, description of invariant measures (both finite and $\sigma$-finite) over the path space of the generalized Bratteli diagrams, and correspondence between tail invariant measures to shift-invariant measures.

Our present results are motivated by diverse applications. In this connection, we stress connections between our present focus on path-space analysis from symbolic dynamics on the one hand and a variety of applied topics on the other. The list includes both direct and indirect interrelationships. Moreover, these connections have increased over time, see e.g., \cite{Beal_Berstel_Eilers_2021}, \cite{Jeandel_2016}, \cite{Banerjee_McPhee_2016} for a partial list. Some of these connections include theoretical computer science, combinatorics, search algorithms, substitution algorithms, analysis on graphs, graph neural network models, theoretical probability, thermodynamics, and empirical behavior of differential equations from big data. Consequently, results from symbolic dynamics have come to play an increasingly important role in applied mathematics.

Substitution dynamical systems defined on a finite alphabet have been extensively studied 
for the last decades. There are several books covering the recent developments, see e.g. \cite{Fogg2002}, \cite{Queffelec_2010}. The connection of substitution dynamical systems with ``finite'' (standard) Bratteli diagrams is well known, see the papers \cite{Durand_Host_Skau_1999}, 
\cite{Forrest_1997}, \cite{Durand_2010} for the case of minimal dynamical systems and   \cite{Bezuglyi_Kwiatkowski_Medynets_2009}, 
\cite{B_K_M_S_2010} for aperiodic substitutions.   Roughly speaking, any
nontrivial substitution dynamical system is represented by a corresponding
Vershik map acting on the path space of a stationary Bratteli diagram. 

Very few results are known about substitution dynamical systems on an 
infinite alphabet. We were motivated by \cite{Ferenczi_2006} where the 
author carefully considered the ``square drunken man substitution''
$n \to (n-2)nn (n+2), n \in \Z$. 

We shall study substitution dynamical systems, stressing the dichotomy of finite alphabet vs infinite alphabet. Recently there has been a renewed interest in the study of dynamical systems on an infinite alphabet. In this context, we refer to the papers on Markov shifts and substitutions: \cite{Takahasi2020}, \cite{OttTomfordeWillis2014},
\cite{Jaerisch2014}, \cite{Gray2011}, \cite{GurevichShevchenko1998},
\cite{RowlandYassawi2017}, \cite{Dombek2012}, \cite{Ferenczi_2006},
\cite{Mauduit2006}. Dynamical systems on infinite alphabets have obvious 
particularities: the underlying space is not compact. Some authors use
the compactification of the space
 \cite{Frank_Sadun_2014}, \cite{Durand_Ormes_Petite_2018} and \cite{Manibo_Rust_Walton_2021}.  In the current work, our focus is on the study of substitution dynamical systems over countable infinite alphabets \textit{without} the assumption of compactification. Hence, the associated subshift should be viewed as a Borel dynamical system (in fact, it is a homeomorphism of a Polish space). 

In the study of substitution dynamical systems over finite alphabets, the notion of \textit{recognizability} plays an important role. The term was first coined by Martin \cite{Martin_1973}. This notion has been studied by Moss\'{e} \cite{Mosse_1992,Mosse_1996}. She showed that primitive aperiodic substitutions on finite alphabets are recognizable. Bezuglyi, Kwiatkowski, and Medynets \cite{Bezuglyi_Kwiatkowski_Medynets_2009} extended Moss\'{e}'s result by relaxing the requirement of primitivity. They showed that all aperiodic substitutions on finite alphabets are recognizable. Berthé, Steiner, Thuswaldner, and Yassawi \cite{Berthe_Steiner_Thuswaldner_Yassawi_2019} studied recognizability (and also eventual recognizability) for sequences of morphisms that define an S-adic shift. The notion of recognizability has been studied in word combinatorics under the name circularity (see \cite{Cassaigne_1994,Klouda_Starosta_2019,Mignosi_Seebold_1993}) and in the study of self-similar tilings as unique composition property (see \cite{Solomyak_1998}). 

For substitutions over infinite alphabets, we need a property similar to recognizability. S. Ferenczi \cite{Ferenczi_2006} introduced a notion  of a  \textit{left determined substitution}, which is
stronger than recognizability (see Definition \ref{left det}).  In this paper, we show that a left determined substitution admits its realization as a Vershik map on a  stationary generalized Bratteli diagram, see Definition \ref{def GBD}. We prove in this paper the following 
result (see Sections \ref{sect left determ} and
\ref{Sec BBD} for necessary definitions).

\begin{theorem}\label{Main_1} Let $\sigma $ be a bounded size left determined substitution on a countably infinite alphabet and $(X_\sigma, T)$ be the corresponding subshift. Then there exists a stationary ordered generalized Bratteli diagram $B = (V, E, \geq)$ and a Vershik map $\varphi : Y_B \rightarrow Y_B$ such that $(X_\sigma, T)$ is isomorphic to $(Y_B, \varphi)$.

\end{theorem}

Our focus on the infinite case in turn is motivated by specific applications. First note that classes of dynamical systems arising from substitutions over a finite alphabet have already proved to play a central role within the wider subject of dynamical systems, and more specifically in the study of symbolic dynamics. One reason for this is the role substitution dynamics play in diverse applications. The other reason is their amenability to use in random walk models with specified transition matrices. The latter in turn are made precise with the use of combinatorial diagrams, more specifically, Bratteli diagrams; or Bratteli-Vershik diagrams.

We recall that Bratteli diagrams were first introduced by O. Bratteli in his seminal paper \cite{Bratteli_1972}. A Bratteli diagram is a combinatorial structure; specifically, a graph composed of vertex levels  (``levels''), and an associated system of unoriented edges, edges between vertices having levels differing by one: each edge only links vertices, source and range, in successive levels, hence loops are excluded. They have turned out to provide diverse and very powerful ``models'' for many mathematical structures. For example approximately finite-dimensional $C^*$ algebras (also called AF-algebras), i.e., inductive limits of systems of matrix algebras. AF-algebras are important in the field of operator algebras and their applications to quantum physics. Moreover, Bratteli diagrams are used in combinatorics, and in computation, for example in 
an analysis of directed systems of fast Fourier transforms, arising for example in big data models.

The first usage of Bratteli diagrams  can be found in  \cite{Vershik_1981}) 
(under a different name) where the author represented any ergodic transformation of probability measure space as a map (now known as a Vershik map) acting on the path space of a Bratteli diagram. His work has been of immense importance as it provides transparent models for ergodic dynamical systems. Then  in the 1990's, Giordano, Herman, Putnam, and Skau (see \cite{HermanPutnamSkau_1992}; \cite{Giordano_Putnam_Skau_1995}) and 
Glasner and Weiss \cite{GlasnerWeiss_1995} successfully extended this work to minimal Cantor dynamical systems. They constructed Bratteli-Vershik models for minimal Cantor systems. As an application, they classified minimal Cantor systems up to orbit equivalence. More results about applications of Bratteli 
diagrams for constructions of models in Cantor dynamics can be found in 
\cite{Medynets_2006},  \cite{DownarowiczKarpel_2019}, 
\cite{Shimomura_2020}. 

In many cases, the properties of a Bratteli diagram are determined by the 
properties of incidence matrices.  They show transitions between neighboring levels.  If the same incidence matrix is used at each level, then it is said 
 that the Bratteli diagram is stationary. A key ingredient in a systematic study of the corresponding dynamical systems is the use of the Perron-Frobenius theorem for the incidence matrices.    As follows from the outline above, studying substitution dynamical systems arising from countably infinite alphabets leads to new challenges. For example, now the incidence matrices are infinite by infinite matrices. Hence the classical Perron-Frobenius theorem must now be adapted to infinite matrices. We use the book 
  \cite{Kitchens1998} as our
 main source for the  Perron-Frobenius theory. 
 
 For infinite recurrent incidence matrices, the generalized Perron-Frobenius theorem gives us a tool to find explicitly tail invariant measures (in other
 words, invariant measures for substitution dynamical systems). 
For the applications at hand, we show that it is possible to arrange that the substitutions are specified by infinite but banded matrices. An infinite banded matrix has its entries supported in a finite-width band around the diagonal. Jacobi-matrices (used in the study of special functions) are perhaps the best known such infinite banded matrices. Our main result is as follows:

 \begin{theorem}\label{Main_2} Let $\sigma $ be a bounded size left determined substitution on a countably infinite alphabet $\A$ and $(X_\sigma, T)$ be the corresponding subshift. Assume that the countably infinite substitution matrix $M$ is irreducible, aperiodic, and recurrent. Then

\vspace{2mm}

\noindent $(1)$  there exists a  shift-invariant measure $\nu$ on $X_\sigma$;

\vspace{2mm}

\noindent $(2)$ the measure $\nu$ is finite if and only if the Perron-Frobenius left eigenvector $\ell = (\ell_i)_{i \in \Z}$ of the matrix $M$ has the property $\underset{i \in \Z}{\sum} \ell_i < \infty$. 

\end{theorem}

The \textit{outline} of the paper is as follows. Section $2$ consists of definitions and preliminary information regarding substitution dynamical systems on infinite alphabets and Borel dynamics. In Section $3$, we discuss crucial properties of \textit{left determined} substitutions. In Section $4$, we construct two versions of Kakutani-Rokhlin towers for subshifts associated with left determined substitution on countably infinite alphabets. In Section $5$, we study the subshift associated with left determined substitution on countably infinite alphabets as Borel dynamical systems. Section $6$ is dedicated to the construction of generalized Bratteli diagrams for such subshifts. In Section $7$, using the stationary generalized Bratteli-Vershik model, we provide an explicit formula for a shift-invariant measure. In the last section (Section $8$), we provide some examples of substitutions on infinite alphabets, corresponding Bratteli-Vershik models, and expressions for tail-invariant measures.  

\section{Preliminaries}\label{Sect Prelim}

\subsection{Substitutions on infinite alphabet} Let $\A$ be a countably 
infinite set, called an \textit{alphabet}; its elements are called \textit{letters}. 
A \textit{word} $w = a_1a_2\cdots a_n$ is a finite string of elements in $\A$. For $n 
\in \N$, we denote by $\A^n$ the set of all words of length $n$. Then
$\A^* = \bigcup_{n\geq 0} \A^n$ is the set of all words on alphabet $\A$ including
 the empty word. A word $w$ is said to be the \textit{prefix} of a word $u$ if 
 $u = ww'$ for some word $w'$. Similarly, we say that $w'$ is the 
 \textit{suffix} of $u$ if $u = ww'$. We denote by $\A^\Z$ the set of all 
 bi-infinite
sequence $(x_i)_{i \in \Z}$ over $\A$ and endow it with the topology induced by 
the metric: 
\begin{equation}\label{metric}
  d(x,y) : = 2^{-\textrm{inf}\, \{|i|: \,\, x_i \,\neq\, y_i\}} \,\, \mathrm{for}\,\, x,y \in 
  \A^{\Z}.
\end{equation} Then the topological space $(\A^\Z, d)$ is a zero-dimensional Polish space. A word $w =a_1a_2\cdots a_n$ is said to occur at the $j$-th place in the infinite sequence 
$x \in \A^{\Z}$ if $x_j = a_1, ..., x_{j+n-1} = a_n$. In this case, we also say that 
$w$ is a \textit{factor} of $x$. Similarly, we will denote by $x_{[p,q]}$ ; $p,q \in \Z$ , $q>p$, the word of length $q-p+1$ occurring at $p$-th place in $x$. 
Given an infinite word $x\in \A^{\Z}$, let  
$\La_n(x)$ denote the set of all factors of $x$ of length $n$.
Then the \textit{language} $\La(x)$ of $x$
 is the set of all factors of $x$, i.e., $\La(x) = \underset{n \in \N}{\bigcup} \La_n(x)$.

\begin{definition}\label{inf sub} A \textit{substitution} $\sigma$ on a countably infinite set $\A$ is an injective map from $\A$ to $\A^+$ (the set of non-empty finite 
 words on $\A$), which associates to every letter $a \in A$ a finite word 
 $\sigma (a)  \in \A^+$. The length of $\sigma(a)$ is denoted  by $h_a := |
 \sigma(a)| $. We assume that $h_a \geq 2$ for all $a \in \A$.
 
For a substitution $\sigma : \A \to \A^+$, the infinite matrix $M = (m_{ab})$, $a, b \in \A$, where $m_{ab}$ is the number of occurrences of the letter $b$ in the word $\sigma(a)$, is called the \textit{substitution matrix} of $\sigma$.
\end{definition}

We extend the substitution $\sigma$ to $\A^*$ by  concatenation, i.e., for
 $w = a_1a_2\cdots a_n \in \A^+$, put $\sigma (w) = \sigma(a_1)\sigma(a_2)\cdots
 \sigma(a_n)$ (and for empty word $\emptyset$, put $\sigma(\emptyset) = 
 \emptyset$). We iterate the substitution $\sigma$ on $\A$ in the usual way by 
 putting $\sigma^i = \sigma \circ \sigma^{i-1}$.

\begin{remark} (1) We included in Definition \ref{inf sub} two properties of 
a substitution $\sigma$: injectivity on letters and growth of word length, 
$|\sigma^n(a)| \geq 2^n$. In other words, we consider a class of substitutions
on $\A$ satisfying these properties. 
This assumption is made to avoid some pathological cases. It will be
 clear from our main results why we need these assumptions.

(2)  The fact that $\A$ is an infinite alphabet gives more possibilities to consider various properties of substitutions $\sigma$ on $\A$ in comparison with
 substitutions on a finite alphabet. Conversely, some properties of finite 
 substitutions
cannot be defined for substitutions on infinite alphabets. In particular, 
there are no primitive substitutions on an infinite alphabet assuming the finiteness of words. 

\end{remark} 

In the definition below, we introduce a class of substitutions of \textit{bounded 
size}. This definition makes the behavior of $\sigma$ more predictable. 

In some cases, it is convenient to identify the set $\A$ with the set of integers 
$\Z$ or the set of natural numbers $\N$. We will denote by $\N_0$ the set $\N \cup 0$.

\begin{definition}\label{Bdd size} We say that $\sigma$ is of \textit{bounded length} if there exists an integer $L \geq 2$ such that for every $a \in \A$, $|\sigma (a)| \leq L$. In the case when, for every $a \in A$, $|\sigma (a)| = L$, 
we say that $\sigma$ is of \textit{constant length}. 

Identifying $\A$ with $\Z$, we say that $\sigma : n \to \sigma(n) $, $n \in \Z$ 
is of \textit{bounded size}, if it is of bounded length and 
there exists a positive integer $t$ (independent of $n$) such that for every $n \in \Z$, if $m \in \sigma (n)$, then $m \in \{n-t,...,n,..., n+t\}$. We call  $t \in \N$ the \textit{size} of $\sigma$ where $t$ is taken minimal possible. 

\end{definition}

\begin{remark} Note that if $\sigma$ is of bounded length (respectively, of 
constant length),
 then the corresponding substitution matrix $M$ has bounded row sum property 
 (respectively, equal row sum property). If $\sigma$ is of bounded size, then
  the  substitution matrix $M$ is a \textit{band matrix} (i.e., non-zero 
 entries are confined to a diagonal band of width $2t +1$), and it has bounded row sum property. 
\end{remark}

 The  \textit{left shift} $T : \A^\Z \rightarrow \A^\Z$ is defined by  $(Tx)_k = x_{k +1}$ for all $k\in \Z$. Recall that a \textit{subshift} $(X, T)$ is a dynamical system where $X$ is a closed shift-invariant subset of $\A^{\Z}$. We define the \textit{language} $\La_{X}$ of a subshift $(X,T)$ as $\La_{X} = \underset{x \in X}{\bigcup} \La(x)$. On the other hand, let $\La$ be a set of finite words on a countably infinite alphabet $\A$ that is closed under taking subwords. We denote by $X_{\La}$ the set of bi-infinite sequences such that all their subwords belong to $\La$. This defines a subshift $(X_{\La}, T)$ which is called the \textit{subshift associated with} $\La$. Interesting examples of such languages are generated using substitutions as defined below.
 
 \begin{definition} \label{def subs ds}
 We define the \textit{language of a substitution} $\sigma$ on a countably infinite alphabet $\A$ by setting
\begin{center}
    $\La_{\sigma} = \{\textrm{factors of }\sigma^n (a):\  \textrm{for some} \,\, n
    \geq 0,\  a \in \A\} .$
\end{center} 
For $n\in \N$, we will denote by $\La_{\sigma}(n)$ all words of length $n$ in the language $\La_\sigma$. Consider the subset $X_\sigma : =  \{x \in \A^\Z : \La(x) \subset \La_\sigma\} \subset \A^\Z$. Then $X_\sigma$ is a closed subset of $\A^{\Z}$ which is invariant with respect to $T$. 
We call $(X_\sigma, T)$ the \textit{subshift} on $\A$ associated with the substitution $\sigma$  or a \textit{substitution dynamical system}.   
\end{definition} 

Every finite word $\ol x = (x_{m}, ... , x_n)$, where $m \leq n \in \Z$, determines a \textit{cylinder set} $[\ol x]_{(m,n)}$ in $X_\sigma$ of length $n-m+1$ :
\begin{equation}\label{cy}
    [\ol x]_{(m,n)} := \{y = (y_i) \in X_{\sigma} : y_{m}=x_{m},..., y_n = x_n\}.
\end{equation}
The zero-dimensional topology on $X_\sigma$ inherited from $\A^{\Z}$ 
is generated by the collection of all cylinder sets. 
Note that these cylinder sets also generate the $\sigma$-algebra\footnote{The letter $\sigma$ used here is a 
traditional notation which is not related to a substitution $\sigma$.} of Borel sets in $X_{\sigma}$. Later, we will work with cylinder sets of the form $[\ol x]_{(0,n)}$, $n\in \N_0$, of length $n+1$
\begin{equation}\label{cy1}
    [\ol x]_{(0,n)} := \{y = (y_i) \in X_{\sigma} : y_0 = x_0, ..., y_n = x_n\}.
\end{equation} For simplicity, when $n$ is known by context we will denote $[\ol x]_{(0,n)}$ by $[\ol x]$. 

\subsection{Borel dynamical system} In Section \ref{Sec sub-BD}, we will 
interpret the subshift on a countably infinite alphabet as a \textit{Borel dynamical system}. Below we provide some definitions from Borel dynamics and 
descriptive set theory.

Let $X$ be a separable completely metrizable topological space (also 
called a \textit{Polish space}), and let $\B$ be the 
$\sigma$-algebra of Borel sets generated by 
open sets in $X$. Then the pair $(X,\B)$ is called a \textit{standard Borel 
space}. Any two uncountable standard Borel spaces
are Borel isomorphic. 

For a standard Borel space $(X, \B)$, a one-to-one Borel map $T$ of $X $ onto 
itself is called a Borel automorphism of $(X, \B)$. Let $Aut(X, \B)$ denote the 
group of all Borel automorphisms of $(X, \B)$.
Any subgroup $\Gamma$ of $Aut(X, \mathcal B)$ is 
called a \textit{Borel automorphism group}, and the pair $(X,\Gamma)$ is 
referred to as a \textit{Borel dynamical system.} In this paper, we will only work 
with groups $\Gamma$ generated by a single automorphism $T$ of $(X,\B)$. 
A Borel automorphism is called \textit{free} if $Tx \neq x$ for every $x \in X$.
If the orbit $Orb_T(x) = \{T^n x : n \in \Z\}$ is infinite for every $x$, then $T$
is called \textit{aperiodic.}

An equivalence relation $E$ on 
$(X, \mathcal{B})$ is called \textit{Borel} if $E$ is a Borel subset of the 
product space $X \times X$. It is said that an equivalence relation $E$ on $(X, 
\mathcal{B})$ is a \textit{countable equivalence relation (CBER)} (or aperiodic) if the  equivalence class 
$[x]_E := \{y \in X : (x,y) \in E\}$ is  countable  for every $x \in X$. 
Let $B \in \B$ be a Borel set, then the \textit{saturation of $B$} with respect 
to a CBER $E$ on $(X,\B)$ is the set $[B]_E$ containing entire equivalence 
classes $[x]_E$ for every $x \in B$. 

In the study of Borel dynamical systems, the theory of countable Borel 
equivalence relations  on $(X, \mathcal{B})$ plays an important role 
as it 
provides a link between descriptive set theory and Borel dynamics 
(see Theorem \ref{FM} below).  We refer the reader to \cite{Kechris_2019} 
for an up-to-date survey of the theory of countable Borel equivalence relations. 
We refer to \cite{DoughertyJacksonKechris_1994}, \cite{BeckerKechris_1996}, 
\cite{Kechris_1995}, \cite{Hjorth_2000}, \cite{JacksonKechrisLouveau_2002}, 
\cite{KechrisMiller_2004}, \cite{Nadkarni_2013}, \cite{Varadarajan_1963}, and 
\cite{Weiss_1984}, where the reader can find connections between the study of 
orbit equivalence in the context of dynamical systems and descriptive set 
theory.

Let $\Gamma \subset Aut(X, \B)$ be a countable Borel automorphism group 
of $(X,\B)$. Then 
the \textit{orbit equivalence relation} $E_X(\Gamma)$ generated by 
$\Gamma$  on $X$ is 
$$
E_X(\Gamma) = \{(x,y) \in X \times X: x=\gamma y
  \  \mbox{for\ some}\ \gamma \in \Gamma\}.
$$ 

The following theorem shows that all CBERs come from Borel actions of 
countable groups.

\begin{theorem} [Feldman--Moore \cite{FeldmanMoore_1977}] \label{FM}  Let 
$E$ be a
 countable Borel equivalence relation on a standard Borel space $(X,
 \mathcal{B})$. 
Then there is a countable group $\Gamma$ of Borel automorphisms of 
$(X,\mathcal{B})$ such that $E = E_X(\Gamma)$.
\end{theorem} 

\begin{definition}\label{section}
A Borel set $C$ is called a \textit{complete section} for an equivalence relation 
$E$ on $(X,\B)$ if every $E$-class intersects $C$. In other words, 
$[C]_E = 
X$.
\end{definition} 

We can also interpret this definition in terms of a Borel automorphism as 
 follows: given $T \in Aut(X, \B)$, a Borel set $C \subset X$ is called a 
 \textit{complete section} (or simply a $T$-\textit{section}) if every $T$-orbit 
 meets $C$ at least once. In other words,  $X = \underset{i\in \Z}{\bigcup} \, 
 T^i C$.

\section{Left determined substitutions} \label{sect
left determ}
Let $\sigma$ be a substitution on $\A$ as in Definition \ref{inf sub}. Then
$\sigma $ defines a map, $\sigma : X_\sigma \to X_\sigma$ where, for 
$x = (x_i)\in X_\sigma$, one takes $\sigma (x) = (\sigma(x_i))$.
We use the same notation $\sigma$ for the map defined on $X_\sigma$. 
It will be clear from the context whether we consider the substitution or the above map. We recall the definition of \textit{recognizability}.

\begin{definition}\label{recog} Let $\sigma$ be a substitution on a countably infinite alphabet $\A$ satisfying Definition \ref{inf sub}.   Let $(X_{\sigma},T)$
 be the corresponding subshift. It is said that $\sigma$ is \textit{recognizable} 
 if for each $x = (x_i)_{i\in \Z} \in X_\sigma$ there exists a unique $y = (y_i)_{i\in \Z} \in X_\sigma$ and unique $k \in \{0, ..., |\sigma(y_0)|-1\}$ such that 
\begin{equation}\label{rec1}
    x = T^k \sigma (y).
\end{equation}
\end{definition}

For the study of substitutions over countable infinite alphabets, Ferenczi 
\cite{Ferenczi_2006} coined the notion of  a \textit{left determined 
substitution} (Definition \ref{left det}). This notion is stronger than recognizability, and it will be a key notion of the current work. We will show 
below that a 
left determined substitution is recognizable (see 
Corollary \ref{rec}).
 
\begin{definition}[\cite{Ferenczi_2006}] \label{left det} A substitution $\sigma$ on a countably infinite alphabet $\A$ is 
called  \textit{left determined} if there exists $N_{\sigma} \in \N$  such that any word $w \in \La_{\sigma}$ of length at least $N_{\sigma}$ has a unique 
decomposition $w = w_1w_2 \cdots w_s$ such that  each $w_i = \sigma (a_i)$ 
for some unique $a_i \in \A$, except that $w_1$ may be a suffix of 
$\sigma (a_1)$ and $w_s$ may be a prefix of $\sigma (a_s)$. 
We will call $w = w_1w_2 \cdots w_s$  
 the \textit{unique $\sigma$-decomposition} of $w$.
\end{definition} 

\textit{Convention.} Dealing with a left determined 
substitution, we will have to use estimates for the 
length of a word. We will assume that $N_\sigma$ in 
Definition \ref{left det} is always chosen minimal 
possible. 

For example, it was shown in \cite{Ferenczi_2006} that the \textit{squared drunken man substitution} 
$$
n \mapsto (n-2)\,n\,n\,(n+2), \,\, n \in 2\Z, \ 
$$ 
is left determined: any word of length at least $3$ has a unique 
$\sigma$-decomposition. We will give more examples of left determined 
substitutions in Section \ref{Sec ex}. 

For this new and more general framework, our first result below shows that every left determined substitution has a natural realization as a 
homeomorphism of a metric space.

\begin{theorem}\label{homeo} Let $\sigma $ be a left determined substitution on a countably infinite alphabet $\A$, then $\sigma$ is a homeomorphism from $X_{\sigma}$ to $\sigma(X_{\sigma})$. 

\end{theorem}

\begin{proof} Note that by definition, $\sigma : \A \rightarrow \A^{+}$ is 
injective and $\sigma (X_{\sigma}) \subset X_{\sigma}$. To see that the map 
$\sigma : X_{\sigma} \rightarrow X_{\sigma}$ is injective, we assume that
there are distinct $x,y \in X_{\sigma}$ such that $\sigma(x) = \sigma(y) = z \in 
X_{\sigma}$. Consider the word $w$ of length $2\ell + 1 \geq N_{\sigma}$  
given by $w = z_{-\ell} \cdots z_0 \cdots z_\ell$ where $z = (z_i)$. Since $
\sigma$ is left determined, $w$ has the unique $\sigma$-decomposition of the 
form
\begin{equation*}
   w = w_{-p}w_{-p+1}\cdots w_0\cdots w_{s-1}w_s = w_{-p}\sigma (a_{-p+1})
   \cdots\sigma (a_0)\cdots \sigma (a_{s-1}) w_s
\end{equation*} 
where $a_j \in \A\,,\, j \in \{-p,..., 0, ...,s\}$, and $p,s$ are positive integers. 
Here 
we chose the labeling $\{-p,..., 0, ...,s\}$ such that $z_0 \in \sigma (a_0)$. Also, 
$w_{-p}$ and $w_s$ are either empty words or a suffix of $\sigma (a_{-p})$ and a prefix of $\sigma (a_s)$, respectively. Since the finite word $w$ has length at least $N_{\sigma}$, any extension $w'$ of $w$ will also have a unique 
$\sigma$-decomposition that does not change the word $\sigma (a_{-p+1})\cdots\sigma (a_0)\cdots \sigma (a_{s-1})$. 
Continuing this procedure, we conclude that  $z$ (as a bi-infinite extension of $w$) also has a unique $\sigma$-decomposition. In other words, there exists unique $u = (a_i)_{i\in\Z}$, such that $\sigma(u) = z$. Since $\sigma$ is 
injective on letters from $\A$, we conclude that $x=u=y$. This proves that   $
\sigma : X_{\sigma} \rightarrow X_{\sigma}$ is injective. 

The continuity of $\sigma$ is obvious. To see that $\sigma^{-1}$ is 
continuous observe that, for $j\in \N$, if $\sigma(x)$ and $\sigma(y)$ agree on 
a cylinder set of length at least $2j$ (say $[c]_{(p,p+2j-1)}$ for some $p\in \Z
$), then $x,y$ agree on a cylinder set of length $j$ (say $[c']_{(q,q+j-1)}$ for 
some $q\in \Z$). This claim follows from the fact that if a word $w$
admits a unique $\sigma$-decomposition $w = w_1\sigma(a_2 \cdots 
a_{s-1})w_s$, then the length of $a_2 \cdots a_{s-1}$ tends to infinity as 
$|w| \to \infty$.
\end{proof} 

The following is a direct corollary of Theorem \ref{homeo}. 

\begin{corollary}\label{rec} If $\sigma $ is left determined substitution on a countably infinite alphabet $\A$, then $\sigma$ is recognizable. 
\end{corollary} 

\begin{proof} Since $\sigma$ is injective for each $x = (x_i)_{i\in \Z} \in X_\sigma$, there exists a unique $y = (y_i)_{i\in \Z} \in X_\sigma$ and such that $x = \sigma (y)$. Thus, $\sigma$ satisfies (\ref{rec1}) for $k=0$.  
\end{proof}

Recall that we consider only injective substitutions $\sigma : \A \rightarrow \A^{+}$. In general, $\sigma^n$ needs not be injective. However,  the following proposition shows that, for left determined substitutions, $\sigma^n$ is also injective for all sufficiently large $n$. Our next result offers an explicit uniform estimate of word length for powers of a general left determined substitution.

\begin{theorem}\label{n-inj} Let $\sigma $ be a left determined substitution on a countably infinite alphabet $\A$. Then, for all $k \in \N$ such that $2^{k-1} \geq N_{\sigma}$, the map $\sigma^k : \A \rightarrow \A^{+}$ is injective.
\end{theorem}

\begin{proof} Given a left determined substitution $\sigma$, fix a natural number $N_\sigma$ such that all words of length at least $N_\sigma$ admit 
 unique $\sigma$-decompositions. We show  that  the map  
\begin{equation}\label{inj1}
    \sigma : \underset{n\geq N_{\sigma}}{\bigcup} \La_{\sigma} (n) \rightarrow \A^{+}
\end{equation} 
is injective. 
To see this, assume that there exist distinct $w_1, w_2 \in \underset{n\geq N_{\sigma}}{\bigcup} \La_{\sigma} (n)$ such that 
\begin{equation}\label{injective}
    \sigma (w_1) = \sigma(w_2) = : w.
\end{equation} Since $|w_1|,|w_2| \geq N_\sigma$, we have $|w| > 
N_{\sigma}$. Hence, $w$ must have a unique $\sigma$-decomposition, 
which contradicts (\ref{injective}). 

Next, we show that, for all $k \in \N$ such that $2^{k-1} \geq N_{\sigma}$, the map $\sigma^k : \A \rightarrow \A^{+}$ is injective. Assume for the contrary 
 that there exist $a_1, a_2 \in \A$ such that $a_1 \neq a_2$ but 
\begin{equation} \label{inj2}
    \sigma^k (a_1) = \sigma^k (a_2) = : \bar{w}.
\end{equation} 
We write (\ref{inj2}) as 
\begin{equation}\label{inj3}
    \sigma (\sigma^{k-1} (a_1)) = \sigma (\sigma^{k-1} (a_2)) = : \bar{w}. 
\end{equation}
 Note that for every $a \in \A$, $|\sigma (a)| \geq 2$, implies $|\sigma^{k-1} (a)| \geq 2^{k-1} \geq N_{\sigma}$. Thus, the injectivity of the map in (\ref{inj1}) contradicts (\ref{inj3}). Hence, the map $\sigma^k : \A \rightarrow \A^{+}$ is injective.  
\end{proof}

\begin{theorem}\label{LD1} Let $\sigma $ be a bounded length substitution on a countably infinite alphabet $\A$ and $C = \mathrm{max} \,\, \{|\sigma(a)|\,; \,\,a \in \A\}$. If $\sigma$ is left determined, then $\sigma^n$ is also left determined for all $n > 1$, and $N_{\sigma^n}$ can be chosen as  
$N_{\sigma} C^{n-1}$.
\end{theorem}

\begin{proof} First we prove the theorem for $n =2$. Let $w$ be a word of length at least $N_{\sigma}C$, where $N_\sigma$ is as in Definition \ref{left det}. We will show that $w$ has a unique representation of the form 
\begin{equation*}
  w = p_1 \sigma^2 (b_2) \cdots\sigma^2 (b_{\ell-1}) p_\ell\,;\, b_j \in \A\,,\, j \in \{1,2,...,\ell\},
\end{equation*} 
where $p_1$ and $p_\ell$ are either empty words or suffix of $\sigma^2 (b_1)$ and prefix of $\sigma^2 (b_\ell)$, respectively. Since $\sigma$ is left determined and $|w| > N_{\sigma}$, the word $w$ has a unique representation of the form
\begin{equation}\label{eq for s}
  w = w_1\cdots w_s = w_1\sigma (a_2)\cdots \sigma (a_{s-1}) w_s\,;\, a_i \in \A\,,\, i \in \{1,2,...,s\},
\end{equation} 
where $w_1$ and $w_s$ are either empty words or suffix of $\sigma (a_1)$ and prefix of $\sigma (a_s)$, respectively. We consider the case when $w_1$ and $w_s$ are both non-empty words (other cases are easier and can be similarly proved). 

We claim that relation (\ref{eq for s}) implies  $s\geq N_{\sigma}$. To see this, 
assume towards a contradiction that $s < N_{\sigma}$. Note that for $i \in 
\{2,...,s-1\}$, $2\leq |\sigma (a_i)| \leq C$ and for $i \in \{1,s\}$, $1 \leq |\sigma 
(a_i)| < C$. These inequalities  imply that $|w| <  sC < 
N_{\sigma}C$, which is a contradiction because we took $w$ such that $|w|
\geq N_{\sigma}C$. 

Let 
$$
\sigma(a_1) = a_{(1,1)} a_{(1,2)}\cdots a_{(1,m)}, \,\textrm{where}\,\,\, 2 \leq m \leq C
$$ and 
$$
\sigma(a_s) = a_{(s,1)} a_{(s,2)}\cdots a_{(s,n)}, \,\textrm{where}\,\,\,  2 \leq n \leq C.
$$ 
Since $w_1$ is suffix of $\sigma (a_1)$ and $w_2$ is prefix of $\sigma(a_s)$, we have
$$
w_1 = a_{(1,t)}\cdots a_{(1,m)}, \,\textrm{for some}\,\,\, t \in \{2,...,m\}
$$ 
and 
$$
w_s = a_{(s,1)}\cdots a_{(s,r)}, \,\textrm{for some}\,\,\, r \in \{1,...,n-1\}.
$$ 
Thus, we can rewrite (\ref{eq for s}) as follows
$$
w = a_{(1,t)}\cdots a_{(1,m)} \sigma(a_2\cdots a_{s-1}) a_{(s,1)}\cdots a_{(s,r)}.
$$ 
We complete the suffix (i.e. $w_1$) and prefix (i.e. $w_s$) to get the word 
$w'$ from $w$ as follows:
\begin{equation}\label{w'}
    w' = a_{(1,1)}\cdots a_{(1,t-1)}a_{(1,t)}\cdots a_{(1,m)}\sigma(a_2\cdots a_{s-1})a_{(s,1)}\cdots a_{(s,r)}a_{(s,r+1)}\cdots a_{(s,n)}
\end{equation}
$$ 
= \sigma(a_1)\sigma(a_2\cdots a_{s-1})\sigma(a_s) = \sigma (a_1\cdots a_s).
$$ 
Since $s \geq N_{\sigma}$ and $\sigma$ is left determined, the word $a_1\cdots a_s$ has the unique representation of the form
\begin{equation}\label{ell}
   a_1\cdots a_s = u_1u_2\cdots u_{\ell-1}u_\ell = u_1 \sigma (b_2)\cdots \sigma (b_{\ell-1}) u_\ell,\,;\, b_j \in \A\,,\, j \in \{1,2,...,\ell\}
\end{equation} where $u_1$ and $u_\ell$ are either empty words or suffix of $\sigma (b_1)$ and prefix of $\sigma (b_\ell)$ respectively. Again, we will consider the case when $u_1$ and $u_\ell$ are both non-empty words. We 
have by (\ref{ell})
$$
w' = \sigma (a_1\cdots a_s) = \sigma(u_1 \sigma (b_2)\cdots \sigma (b_{\ell-1}) u_\ell)
$$Let $u_1 = a_1a_2\cdots a_{k_1}$  for some $k_1 \in \{1,...,s\}$. Similarly, let $u_{\ell} = a_{k_2}a_{k_2 +1}\cdots a_s$  for some $k_2 \in \{1,...,s\}$. Since $u_1$ is a suffix of $\sigma (b_1)$, $\sigma(a_1a_2\cdots a_{k_1}) = \sigma(u_1)$ is the suffix of $\sigma^2 (b_1)$.  It follows from  (\ref{w'}) that
\begin{center}
    $a_{(1,1)}\cdots a_{(1,t-1)}a_{(1,t)}\cdots a_{(1,m)} \sigma(a_2)\cdots \sigma(a_{k_1})$ is a suffix of $\sigma^2 (b_1)$.
\end{center} Hence,
\begin{center}
    $p_1 := a_{(1,t)}\cdots a_{(1,m)} \sigma(a_2)\cdots \sigma(a_{k_1})$ is a  suffix of $\sigma^2 (b_1)$.
\end{center} 
Similarly, since $u_{\ell}$ is a prefix of $\sigma (b_\ell)$, the word 
$\sigma(a_{k_2}a_{k_2+1}\cdots a_{s}) = \sigma(u_\ell)$ is a prefix of 
$\sigma^2 (b_\ell)$. Relation (\ref{w'}) implies that
\begin{center}
    $\sigma(a_{k_2}) \sigma(a_{k_2+1})\cdots \sigma (a_{s-1})a_{(s,1)}\cdots a_{(s,r)}a_{(s,r+1)}\cdots a_{(s,n)}$ is a prefix of $\sigma^2 (b_\ell)$.
\end{center} 
Hence,
\begin{center}
    $p_\ell : = \sigma(a_{k_2}) \sigma(a_{k_2+1})\cdots \sigma (a_{s-1})a_{(s,1)}...a_{(s,r)}$ is a prefix of $\sigma^2 (b_\ell)$.
\end{center} 
Thus, we get 
\begin{equation*}
    w = p_1 \sigma^2(b_2)\cdots \sigma^2(b_{\ell-1}) p_\ell
\end{equation*} as needed. We proved that every word of length at least $N_{\sigma}C$ has a unique $\sigma$-decomposition.  
The proof for $n > 2$ is similar.
\end{proof}

\section{Kakutani-Rokhlin towers for infinite substitutions.}\label{sect Rokhlin} In this section, we construct Kakutani-Rokhlin (K-R) towers for subshifts associated with left determined substitutions on countably infinite alphabets. The bases of K-R towers in our construction are 
represented by cylinder sets. We give two versions of K-R towers based on 
the structure of cylinder sets that form the bases of the towers. These constructions are central to our work as they form building blocks for Bratteli-Vershik models of the corresponding subshifts.

Recall that for a finite word $\ol x = (x_{0}, ... , x_n)$, $[\ol x]$ denotes the cylinder set of the form $[\ol x]_{(0,n)}$, $n\in \N_0$ (see (\ref{cy1})). We should mention that Theorem \ref{Rok 1} below
 is a generalization of  Lemma $2.8$ from \cite{Ferenczi_2006}, where the author worked with the squared drunken man substitution which has 
 constant length.

\begin{theorem}\label{Rok 1} Let $\sigma $ be a bounded length left determined substitution on a countably infinite alphabet $\A$, and let 
$(X_{\sigma},T)$ be the corresponding subshift. Then, for any $n \in \N$, there exists a partition of $X_\sigma$ into Kakutani-Rokhlin towers given by
\begin{equation} \label{eq Rokh1}
    X_{\sigma} = \underset{a \in \A}{\bigsqcup} \,\,\, \underset{k = 0}{\overset{h^{(n)}_a  - 1}{\bigsqcup}} \, T^k [\sigma^n (a)]
\end{equation} 
where $ h^{(n)}_a = |\sigma^n (a)|$, $a \in \A$.
\end{theorem}

\begin{proof} First, we prove this theorem for $n = 1$. 
Since $\sigma$ is left determined, there exists $N_{\sigma} \in \N$ such that every word of length greater than or equal to $N_{\sigma}$ has a unique decomposition. If $N_{\sigma}$ is odd, put $\ell  = \dfrac{N_{\sigma}-1}{2}$, otherwise put $\ell  = \dfrac{N_{\sigma}}{2}$. We take arbitrary $x$ in $X_\sigma$ and show that there exists $a \in \A$ such
that 
$$
x \in \underset{k = 0}{\overset{h_a  - 1}{\bigsqcup}} \, T^k [\sigma (a)].
$$
For $x = \{x_i\}_{i \in \Z} \in X_\sigma$, consider the word $w =  x_{-\ell} \cdots
x_0\cdots x_\ell$ of length 
$2\ell + 1 = N_{\sigma}$.  Since $\sigma$ is left determined,
 $w$ has a unique representation of the form
\begin{equation}\label{w_dec}
   w = w_{-p}w_{-p+1}\cdots w_0\cdots w_{s-1}w_s = w_{-p}\sigma (a_{-p+1})\cdots \sigma (a_0)\cdots  \sigma (a_{s-1}) w_s
\end{equation} 
for $a_j \in \A\,,\, j \in \{-p,...0...,s\}$, where $p,s$ are positive integers. Here we 
choose the labeling $\{-p,...0...,s\}$ such that $x_0 \in \sigma (a_0)$. Also, by 
definition of  a left determined substitution, $w_{-p}$ and $w_s$ are either 
empty words or suffix of $\sigma (a_{-p})$ and prefix of $\sigma (a_s)$,  
respectively. 

We denote $h_0 : = h_{a_0} = |\sigma (a_0)|$ and set $\sigma (a_0) = b_{1}\cdots b_{h_0}$, where $b_{i} \in \A$ for $i \in \{1, ...,h_0\}$. By uniqueness of the representation in (\ref{w_dec}), $a_0$ is uniquely 
determined. Also $x_0 = b_{j}$ for some $j \in \{1, ...,h_0\}$. We define 
\begin{equation*}
    y = (y_i),\,  y_i = x_{i-(j-1)},\, \mathrm{for}\,\mathrm{all}\,\, i \in \Z.
\end{equation*} Set $j-1 = k$, then $T^{-k} x = y$. Since $y \in [\sigma (a_0)]$, we obtain that $x \in T^{k} [\sigma (a_0)] $ for some $k \in \{0,...,h_0 -1\}$. Hence, we have
\begin{equation}\label{tow1}
    X_{\sigma} = \underset{a \in \A}{\bigsqcup} \,\,\, \underset{k = 0}{\overset{h_a  - 1}{\bigsqcup}} \, T^k [\sigma (a)]
\end{equation} 
where $h_a = |\sigma(a)|$, for $a \in \A$. We claim that the two unions in (\ref{tow1}) are disjoint. To see this assume that the first union is not disjoint. Hence there exists $x \in X_{\sigma}$ such that $x \in T^{k_i} [\sigma(a_i)] \cap T^{k_j} [\sigma(a_j)]$ for $a_i \neq a_j \in \A$, $k_i \in \{0,...,h_{a_i} -1\}$ and $k_j \in \{0,...,h_{a_j} -1\}$. This implies that the word $w =  x_{-\ell} \cdots
x_0\cdots x_\ell$ has two different decompositions (of the form (\ref{w_dec})) one with $a_0 = a_i$ and another with $a_0 = a_j$. This contradicts the fact that $\sigma$ is left determined. Similarly, assume that the second union in  (\ref{tow1}) is not disjoint. Hence there exists $x \in X_{\sigma}$ such that $x \in T^{k_i} [\sigma(a)] \cap T^{k_j} [\sigma(a)]$, for $k_i\neq k_j \in \{0,...,h_{a} -1\}$. In other words $ [\sigma(a)] \cap T^{k_j - k_i} [\sigma(a)] \neq \emptyset$. This again implies that word $w =  x_{-\ell} \cdots
x_0\cdots x_\ell$ has two different decompositions (of the form (\ref{w_dec})), which is a contradiction. Hence for every $x \in X$ there is a unique $a\in \A$ and $k \in \{0,...,h_{a}-1\}$ such that $x \in T^{k} [\sigma (a)]$. This proves the theorem for $n=1$.

 The statement for $n>1$ is proved similarly. It follows from the 
 fact that  if $\sigma$ is a bounded length left determined substitution, then, 
 for every $n>1$, $\sigma^n$ is also a left determined substitution 
 (Theorem \ref{LD1}), and we can repeat the above proof for $\sigma^n$.
  Therefore,  the decomposition of $X_\sigma$ analogous to (\ref{tow1}) holds. 
  \end{proof}

For completion, we state an analogue of the existence of K-R partition for constant length substitutions (in other words, the substitution matrix has the
equal row sum property) which follows immediately from 
Theorem \ref{Rok 1}.

\begin{corollary}\label{Rok 2}
Let $\sigma $ be a left determined substitution of 
constant length $L$  on a countably infinite alphabet $\A$, and let 
$(X_{\sigma},T)$ be the corresponding subshift. Then, for every $n \in \N$, 
\begin{equation} \label{eq Rokh1.2}
    X_{\sigma} = \underset{a \in \A}{\bigsqcup} \,\,\,\, \underset{k = 0}{\overset{L^n-1}{\bigsqcup}} \, T^k [\sigma^n (a)]. 
\end{equation} 
\end{corollary} 
Note that in Theorem \ref{Rok 1} (respectively Corollary \ref{Rok 2}), we assumed the substitution $\sigma$ was of bounded length (constant length, respectively). In Appendix \ref{App_A}, we provide a version of K-R partitions for subshifts associated with left-determined substitutions on countably infinite alphabets (see Theorem \ref{RK_n}), which does not need the assumption of bounded or constant length.

\section{Substitutions on infinite alphabet as Borel dynamical systems}\label{Sec sub-BD} Let $(X_\sigma,T)$ be the subshift associated with a substitution $\sigma$ on a countably infinite alphabet $\A$. We recall that the space 
$X_\sigma \subset \A^{\Z}$ is equipped with the subspace topology $\omega$ inherited from the metric topology on $\A^{\Z}$ (see (\ref{metric})). Then 
$(X_{\sigma},\omega)$ is a $0$-dimensional Polish space. By $\B$, we denote the sigma-algebra generated by open sets in $X_{\sigma}$ so that 
$(X_{\sigma},\B)$ is a standard Borel space.  The left shift $T$ is a Borel automorphisms, i.e., $T \in Aut(X_{\sigma},\B)$. In the rest of this paper, we will consider the subshift $(X_\sigma, T)$ as a Borel dynamical system. This allows us to interpret subshifts from a different perspective and use some methods of 
Borel dynamics. It is not hard to see that $T$ is, in fact, a homeomorphism of $X_\sigma$. 

\subsection{Nested sequence of complete sections} Our goal is to construct a Bratteli-Vershik model of the subshift $(X_{\sigma}, T)$ associated with a left determined substitution $\sigma$ on a countable alphabet (see Section \ref{Sec BBD}). An important ingredient of this construction is a nested sequence of complete sections (see Definition \ref{section}). In this section, we construct nested sequences of complete sections for subshifts (see Corollary \ref{CS 2}). Since there are many interesting examples of bounded size substitutions on countably infinite alphabets (see Definition \ref{Bdd size}), we will focus on this type 
of substitution in the main body of the paper. This assumption has some direct implications: for example, it allows us to control the cardinality of the intersection of the sequence of complete sections (Proposition \ref{base 2}), which, in turn, helps us to control the cardinality of the sets of minimal and maximal paths in the path space of the corresponding Bratteli-Vershik model (see Section \ref{Sec BBD}). In Appendix \ref{App_A} we relax the requirement of bounded size and provide a construction of complete sections associated with such subshifts (see Corollary \ref{CS}).

We recall that the notation $[\cdot]$ is used for a cylinder set of $X_\sigma$.

\begin{proposition}\label{base 2} Let $\sigma$ be a bounded size left determined substitution on a countable alphabet $\A$. We define the sequence of Borel sets $\{A_n\}_{\N_0}$ as follows:
\begin{equation} \label{com sec 2}
    A_n = \underset{a_i \in \A}{\bigsqcup}  [\sigma^n(a_i)], \,\,\, \textrm{for} 
    \,\,n \in \N_0.
\end{equation} 
Then the set $\underset{n \in \N_0}{\bigcap} A_n $ is at most countably infinite. 
\end{proposition}

\begin{proof} For $x \in \underset{n\in \N_0}{\bigcap} A_n$, there exists a 
sequence of letters $(a_n)_{n \in \N}$ such that $x \in [\sigma(a_1)]$, $x \in 
[\sigma^2(a_2)]$,..., $x \in [\sigma^n(a_n)]$,... and so forth. Thus, to find the 
cardinality of $\underset{n \in \N_0}{\bigcap} A_n $, we need to find out how
many sets of the form $\underset{n \in \N_0}{\bigcap} [\sigma^n(a_n)]$ are non-empty, and then determine the cardinality of each such set. 

To determine how many sets can have the above form, pick the
 smallest $k \in \N$ such that $2^{k-2} \geq N_{\sigma}$. Then $|\sigma^k (a_k)| \geq 2^k \geq N_{\sigma}$ and it has a unique $\sigma$-decomposition of the form
\begin{equation}\label{k-rep}
  \sigma^k (a_k) = w_1\sigma (b_2)\cdots \sigma (b_{s-1}) w_s\,;\, b_i \in \A\,,\, i \in \{1,2,...,s\},
\end{equation} 
where $w_1$ and $w_s$ are either empty words or a suffix of $\sigma (b_1)$ and a prefix of $\sigma (b_s)$, respectively. Since 
\begin{equation}\label{k1-rep}
 \sigma(\sigma^{k-1} (a_k)) = w_1\sigma (b_2)\cdots \sigma (b_{s-1}) w_s, 
\end{equation}
 it implies $w_1 = \sigma (b_1)$ and $w_s = \sigma (b_s)$, i.e. 
 \begin{equation}\label{k2-rep}
 \sigma(\sigma^{k-1} (a_k)) = \sigma(b_1b_2\cdots b_s)
\end{equation} 
By Theorem \ref{n-inj}, the map 
$$
\sigma: \underset{n\geq N_{\sigma}}{\bigcup} \La_{\sigma} (n) \rightarrow \A^{+}
$$ 
is injective. Note that $|\sigma^{k-1} (a_k)| \geq 2^{k-1} \geq N_{\sigma}$, hence relation 
(\ref{k2-rep}) implies 
\begin{equation}\label{k3-rep}
 \sigma^{k-1} (a_k) = b_1b_2\cdots b_s.
\end{equation} 
Since we assumed that $2^{k-2} \geq N_{\sigma}$, Theorem \ref{n-inj} implies 
that the map  $\sigma^{k-1} : \A \rightarrow \A^{+}$ is injective. Hence, by 
(\ref{k3-rep}) and the fact that the word $b_1\cdots b_s \in \A^+$ is unique, we 
get that $a_k$ is unique. A similar argument shows the uniqueness of $a_n$ for 
all $n > k$. Since $\sigma$ is of bounded size, once we fix $a_k$, there are $(2t + 1)$ choices for $a_{k-1}$, $(2t + 1)^2$ choices for $a_{k-2}$,...., and $(2t + 
1)^{k-1}$ choices for $a_0$. Here $t \in \N$ is the size of $\sigma$ (see Definition \ref{Bdd size}). In other words, for a fixed $a_k$, there are finitely 
many choices left for terms $a_0$,...,$a_{k-1}$. We have proved the following 
result: If $\underset{n \in \N_0}{\bigcap} [\sigma^n(a_n)] \neq \emptyset$, then 
there exist a $k \in \N$ such that for all $n\geq k$, $a_n$ is uniquely 
determined and there are finitely many choices left for terms $a_0,..., a_{k-1}$. 
It follows from the proved result that there exist countably many 
non-empty sets of the form $\underset{n \in \N_0}{\bigcap} [\sigma^n(a_n)]$. 

To complete the proof of the proposition, we will show the following: If for a 
sequence of letters $(a_n)_{n \in \N_0}$ the set $\underset{n \in \N_0}{\bigcap} 
[\sigma^n(a_n)] $ is non-empty, then it is a singleton set. For this, assume that $x 
\in \underset{n \in \N_0}{\bigcap} [\sigma^n(a_n)]$, and choose $k \in \N$ such 
that $2^{k-1} \geq N_{\sigma}$. We consider the cylinder set $[\sigma^k (a_k)]$ 
containing $x$. By Theorem \ref{n-inj}, the map  $\sigma^{k} : \A \rightarrow \A^{+}$ 
is injective, and this means that the set  $x_{[0,..., |\sigma^k (a_k)| -1]}$ is
uniquely defined. As $k \rightarrow 
\infty$, we see that the string  $\{x_i\}_{i\geq 0}$ is uniquely defined. 
It remains to show  that the same is true for $\{x_i\}_{i < 0}$. Consider
 the point $y = T^{-1}x$. Note that for 
every $n \in \N$, there exists $a'_{n} \in \A$ such that $y \in T^{(|\sigma^n 
(a'_n)| - 1)} [\sigma^n (a'_n)]$. In other words, for every $n \in \N$, $y$ lies on 
the top of a Kakutani-Rokhlin tower (see (\ref{eq Rokh1})), with base $
[\sigma^n (a'_n)]$ for some $a'_{n} \in \A$. Put $z^{(n)} : = T^{-(|\sigma^n 
(a'_n)| - 1)} y$ for each $n \in \N$. Pick $p \in \N$ such that $2^{p-1} \geq 
N_{\sigma}$, then applying the same proof as in (\ref{k-rep}) through 
(\ref{k3-rep}), we get that the word $\sigma^p (a'_{p})$ is uniquely 
determined. Hence the cylinder set $[\sigma^p (a'_{p})]$ is unique. Since 
$z^{(p)} \in [\sigma^p (a'_{p})]$, this fixes the entries $z^{(p)}_{[0,..., |\sigma^p (a'_p)| -1]}$. 
Note that
$$
z^{(p)}_{[0,..., |\sigma^p (a'_p)| -1]} = y_{[-(|\sigma^p (a'_p)| -1),...,0]} = x_{[-(|\sigma^p (a'_p)|),...,-1]}.
$$

Thus, it proves that $x_{[-(|\sigma^p (a'_p)|),...,-1]}$ is uniquely determined. As $p \rightarrow \infty$, the set $\{x_i\}_{i<0}$ is uniquely determined, too.
 Hence, we have shown that if the set $\underset{n \in \N_0}{\bigcap} [\sigma^n(a_n)]$ is non-empty, then it is a singleton set. This completes the proof of the proposition. 
 \end{proof} 

The following corollary is an important tool in the construction of 
Bratteli-Vershik models for substitutions on infinite alphabets (see Section 
\ref{Sec BBD}).  

\begin{corollary}\label{CS 2} 
Let $\sigma $ be a bounded size left determined substitution on countably infinite alphabet $\A$, and let $(X_{\sigma},T)$ be the corresponding subshift. Then the sequence of Borel sets $(A_n)_{n\in \N_0} \subset X_{\sigma}$, defined in (\ref{com sec 2}), has following properties:

\begin{enumerate}[(a)]
\item $X_{\sigma} = A_0 \supset A_1 \supset A_2 \supset A_3\cdots\, $;
\item The countable set $A_{\infty} : = \underset{n\in \N_0}{\bigcap} A_n$ is  a wandering set with respect to $T$;
\item $A_n$ is a complete $T$-section for each $n \in \N_0$;
\item For each $n \in \N_0$, every point in $A_n$ is recurrent, i.e., the 
$T$-orbit of every point returns to $A_n$.
\end{enumerate}

\end{corollary}
\begin{proof} Note that for all $n \in \N$ and $a \in \A$, there exists $b \in \A$ such that, $[\sigma^n (a)] \subset [\sigma^{n-1} (b)]$. To see this, assume that for some $a \in \A$, there exists $t \in \N$ such that $[\sigma^t (a)] \not\subset [\sigma^{t-1} (b)]$ for every $b \in \A$. This implies that there is no letter $b \in \A$, such that $\sigma^{t-1} (b)$ forms the first $|\sigma^{t-1} (b)|$ entries of the word $\sigma^t (a)$. This, in turn, implies that there is no letter $b \in \A$, such that $b$ is the first letter in the word $\sigma(a)$, which is a contradiction. This observation together with the definition of $A_n$ proves
 $(a)$. 
 
 The set $A_{\infty}$ is a countable set, see Proposition \ref{base 2}. Note that $A_{\infty}$, $T(A_{\infty})$, $T^2(A_{\infty})$, ...,$T^{n-1}(A_{\infty})$ are pairwise disjoint for every $n$ because of Proposition \ref{base 2}. Therefore, $A_{\infty}$ is a wandering set. This proves part $(b)$. 
 
 To see that $(c)$ and $(d)$ are true, we remark that, for every $n \in \N$, $A_n$ is the union of bases of Rokhlin towers given in (\ref{eq Rokh1}). Hence, for every $n \in \N$, $A_n$ is a complete section, and every point in $A_n$ is recurrent. \end{proof} 
 
 By Theorem \ref{Rok 1}, for $n \in \N$, $X_{\sigma}$ is decomposed into $T$-towers 
\begin{equation}\label{eq Tow 2}
   \xi_n(a) = \{T^i [\sigma^n (a)] : i = 0,..., (h^n_{a} -1)\},\,\,\,\, a \in \A
\end{equation}  where $h_a = |\sigma (a)|$. In other words, for $a \in \A$, $\xi_n(a)$ denotes a tower of height $h^n_a - 1$ with cylinder set $[\sigma^n (a)]$ as its base. For $n \in \N$, we will denote by
\begin{equation}\label{tower 2}
    \xi_n = \{\xi_n(a) : a \in \A\}
\end{equation} the countable partition of $X_\sigma$. For sake of completion, we denote by $\xi_0$ the following countable partition of $X_{\sigma}$:
$$
\xi_{0} = \underset{a_i \in \A}{\bigsqcup} [a_i].
$$
\begin{remark}\label{refine 2}
Let  $(\xi_n)_{n \in \N_0}$ be the sequence of 
countable partitions defined above. Observe that,

\noindent $(i)$ $\xi_{n+1}$ refines $\xi_n$ for every $n \in \N_0$. To see this, note that for every $n \in \N$, and $a \in \A$, there exist $b \in \A$ such that $[\sigma^n (a)] \subset [\sigma^{n-1} (b)]$;

\noindent $(ii)$  $\bigcup_n \xi_n$ generates the sigma-algebra of Borel sets $\B$;

\noindent $(iii)$  all elements of partitions $\xi_n$, $n\in \N_0$ are clopen in the topology $\omega$. 
\end{remark}


\section{Generalized Bratteli-Vershik models for infinite substitutions.} \label{Sec BBD} 
\subsection{Generalized Bratteli diagrams}\label{GBDs} The notion of a generalized Bratteli diagram is a natural extension of the notion 
of a Bratteli diagram, where each level in the diagram is allowed to be 
countable. It was proved in \cite{BezuglyiDooleyKwiatkowski_2006}
that any aperiodic Borel automorphism can be realized as a Vershik map on the 
path space of a generalized Bratteli diagram, see Theorem \ref{thm borel-bratteli} below. In this section, we will provide an algorithm to construct a  
generalized stationary Bratteli-Vershik model for a subshift associated with a 
bounded size left determined substitution on a countably infinite alphabet.

\begin{definition}\label{def GBD} A \textit{generalized Bratteli diagram} is a 
graded graph $B = (V, E)$ such that the vertex set $V$ and the edge set $E$ 
can be represented as partitions  $V = \bigsqcup_{i=0}^\infty  V_i$ and $E = 
\bigsqcup_{i=0}^\infty  E_i$ satisfying the following properties: 
\vspace{2mm}

\noindent 
$(i)$ The number of vertices at each level $V_i$, $i \in \N_0$, is countably 
infinite (in most cases, we will identify each $V_i$ with either $\Z$ or $\N$). 
For all $i \in 
\N_0$, $E_i$ represents the countable  set of edges between the levels $V_i$ and $V_{i+1}$. 

\vspace{2mm}

\noindent $(ii)$ The range and source maps $r,s : E \rightarrow V$ are defined on the diagram $B$ such that $r(E_i) \subset V_{i+1}$ and $s(E_i) \subset V_{i}$ 
for each $i \in \N_0$. Moreover, for all $v\in V$, $s^{-1}(v)\neq \emptyset $, 
and, for all $v \in V\setminus V_0$, $r^{-1}(v)\neq\emptyset$. 

\vspace{2mm}

\noindent $(iii)$ For every vertex $v \in V \setminus V_0$, $|r^{-1}(v)| < \infty$. Here $|\cdot|$ denotes the cardinality of the set.
\end{definition}

For a vertex $v \in V_i$ and a vertex $w \in V_{i+1}$, we will denote by $E(v, w)$ the set of edges between $v$ and $w$. It follows from Definition \ref{def GBD} that the set $E(v, w)$ is either finite or empty for every fixed pair of vertices 
$(v, w)$. Set $f^{(i)}_{w,v} = |E(v, w)|$ for every $v \in V_i$ and $w \in V_{i+1}$. In such a way,  we associate with a generalized Bratteli diagram $B =(V, E)$ 
a sequence of non-negative infinite matrices $F_i$  ($i \in \N_0$) that are
called the \textit{incidence  matrices} and given by 
$$
F_i = (f^{(i)}_{w,v} : w \in V_{i+1}, v\in V_i),\ \   f^{(i)}_{w,v} 
 \in \N_0.
$$
\begin{remark} The next observations follow directly from Definition \ref{def GBD}:
\vspace{2mm}

\noindent $(1)$ The structure of a generalized Bratteli diagram is completely 
determined by the sequence of incidence matrices $(F_n), \, n \in \N_0$. We 
will write $B = B(F_n)$ to emphasize that the generalized Bratteli diagram $B$ is determined by $(F_n)$. 

\vspace{2mm}

\noindent $(2)$ For each $n \in \N_0$, the matrix $F_n$ has at most finitely many non-zero entries in each row and none of its rows or columns are entirely zero. A column of $F_n$ can have a finite or infinite number of non-zero entries. 

\end{remark}

\begin{definition} Let $B(F_n)$ be a generalized Bratteli diagram such that for every $n \in \N_0$ the matrix $F_n$ is same, i.e. $F_n = F$, $n \in \N_0$, then we call the diagram a \textit{stationary generalized Bratteli diagram}. We will write $B = B(F)$ in this case. 
    
\end{definition}

Figure \ref{fig:kernels} is an example of a generalized Bratteli diagram. Clearly, we can see only a finite part of the diagram as there are countably many levels and each level has countably many vertices. 
\begin{figure}[!htb] 
\centering
  \includegraphics[width=0.75\textwidth, height=0.45\textheight]
  {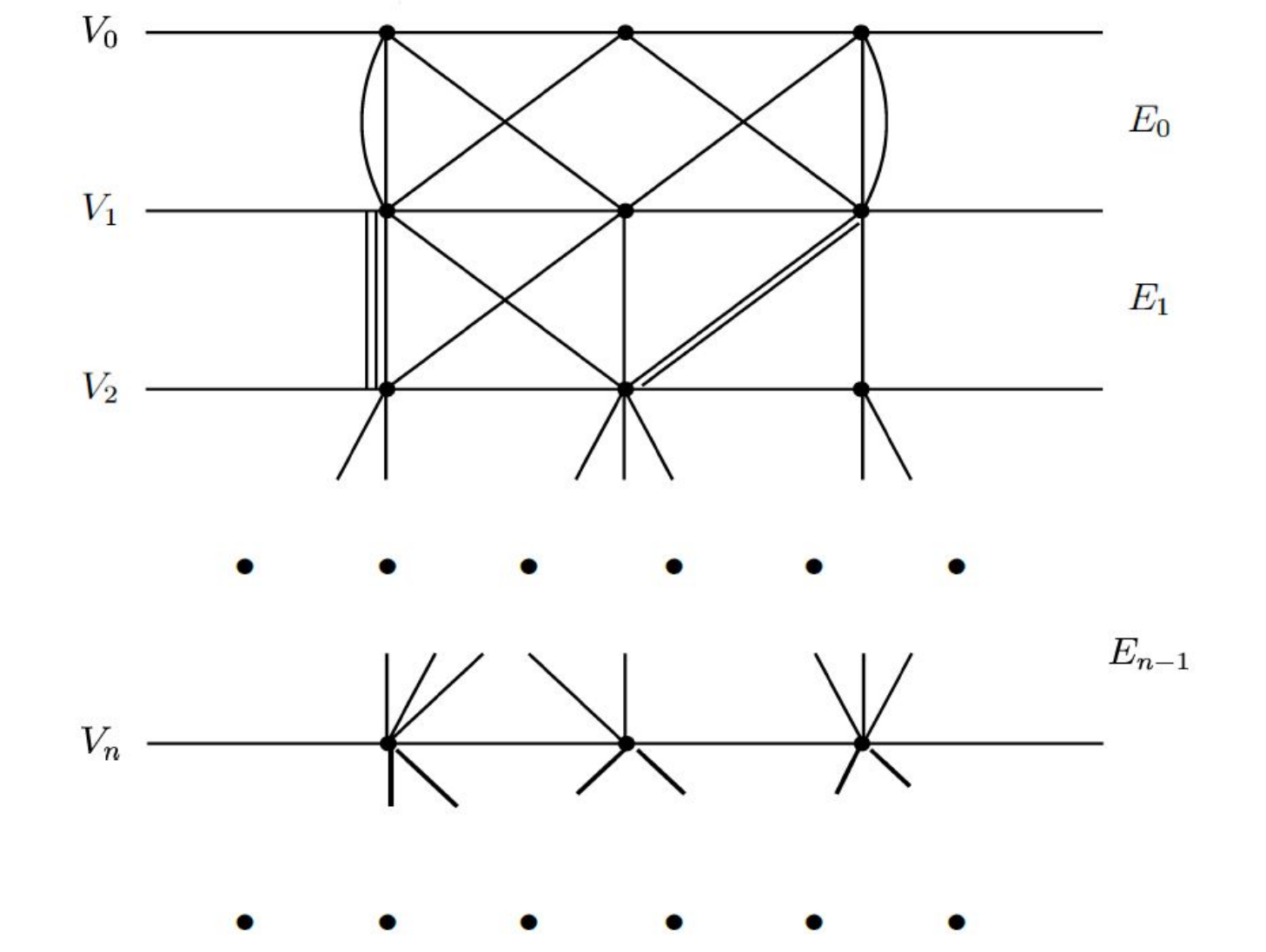}
  \caption{Example of a Bratteli diagram: levels, vertices, and edges
  (see Definition \ref{def GBD})}
  \label{fig:kernels}
  \end{figure}

If the level $V_0$ consists of a single vertex and each set $V_n$ is finite, then we get the usual definition of a Bratteli diagram which was first defined in \cite{Bratteli_1972} and used to model Cantor dynamical systems and classify them with respect to orbit equivalence (see \cite{HermanPutnamSkau_1992}, \cite{Giordano_Putnam_Skau_1995}). 

\begin{definition}\label{def path space}
A finite or infinite path in a generalized Bratteli diagram $B=(V,E)$ is a sequence of edges $(e_i, \, i \in \N_0)$ such that $r(e_i) = s(e_{i+1})$ for each $i$. For a diagram $B=(V, E)$, we will denote by $Y_B$ the set of all infinite paths in $B$. This set is called the \textit{path space} of the generalized Bratteli diagram $B=(V,E)$. If $\ol e = (e_0, ... , e_n)$ is a finite path in $B=(V,E)$, then the set 
\begin{equation}\label{cy_B}
    [\ol e] := \{x = (x_i) \in Y_B : x_0 = e_0, ..., x_n = e_n\}, 
\end{equation} 
is called the \textit{cylinder set} associated with $\ol e$. It is the set containing all infinite paths with the first $n+1$ entries same as in $\ol e$. 

Two paths $x= (x_i)$ and $y=(y_i)$ are called \textit{tail invariant} 
if there exists some $n$ such that $x_i= y_i$ for all $i> n$. This notion defines a 
countable Borel equivalence relation $\mathcal R$ on the path space $Y_B$
which is called the \textit{tail equivalence relation}. A Borel measure (finite or sigma finite) that is invariant under the tail equivalence relation $\mathcal{R}$ is called a \textit{tail invariant} measure.
\end{definition}

 For $v \in V_0$ and $w \in V_n$, $n \in \N$, we will denote the finite path starting at $v$ and ending at $w$ by $\ol e(v, w)$ and the corresponding cylinder set by $[\ol e(v, w)]$. 

\begin{remark} In this remark, we collect several obvious statements about the path space $Y_B$ of a generalized Bratteli diagram $B$.

\vspace{2mm}

\noindent $(1)$ The path space $Y_B$ is a zero-dimensional Polish space with the topology generated by cylinder sets. Moreover, every cylinder set is a clopen set in this topology. This topology coincides with the topology generated by the following metric on $Y_B$: for  $x = (x_i), \, y = (y_i)$ ($x \neq y$), define 

$$
\mathrm{dist}(x, y) = \frac{1}{2^N},\ \ \ N = \min\{i \in \N_0 : x_i 
\neq y_i\}.
$$

\vspace{2mm}

\noindent $(2)$ The path space $Y_B$ of a generalized Bratteli diagram is a standard Borel space, where the Borel structure is generated by clopen (cylinder) sets. In general, $Y_B$ is not locally compact. One can see that 
$Y_B$ is locally compact if every column of $F_n$ has finitely many non-zero
entries. 

\vspace{2mm} 

\noindent $(3)$ We will assume that the path space $Y_B$ of a generalized Bratteli diagram $B$  has no isolated points, i.e., for every infinite path $(x_0, x_1, x_2, ... ) \in Y_B$ and every $n \in \N_0$, there exists a level $m > n$ such that $|s^{-1}(r(x_m))| > 1$. 
\end{remark}

\begin{definition}\label{tele} Given a generalized Bratteli diagram $B = (V,E)$ and a  monotone increasing sequence $(n_k : k \in \N_0), n_0 = 0$, we define a new generalized Bratteli diagram $B' = (V', E')$ as follows: the vertex sets are 
determined by $V'_k = V_{n_k} $, and the edge sets  $E'_k = E_{n_{k+1}-1} \circ ...\circ E_{n_k}$ are formed by finite paths between the levels $V'_k$ and 
$V'_{k+1}$. The diagram $B' = (V', E')$  is called a \textit{telescoping} of the original diagram $B = (V,E)$. The incidence matrices  of $B'$ are the 
products of incidence matrices of $B$: 
$F'_k = F_{n_{k+1}-1} \circ ...\circ F_{n_k}$.
\end{definition}

For a generalized Bratteli diagram $B = (V, E)$, we define a \textit{partial order} $\geq$ on each edge set $E_i$  as follows:  for every $v \in V_{i+1}$, equip the finite set $r^{-1}(v)$ with a \textit{linear order} $\geq$ so that any two edges $e,e' \in E_i$ are comparable if and only if they have the same range, 
$r(e) = r(e')$. Since we do this for each $E_i \, (i \in \N_0)$, we obtain a partial order $\geq$ on the set of finite paths. Using the partial order $\geq$, 
we can define a \textit{partial lexicographic order} on the set of all paths between any two  levels (say  $V_t$ and $V_s$ for $s>t$):
$(e_{t},...,e_{s-1}) > (f_{t},...,f_{s-1})$ if and only if for some $i$ 
with $t \le i <s$, $e_j=f_j$ for $i<j < s$ and $e_i> f_i$. 

Applying this lexicographic order, we can compare any two finite paths in 
$E(V_0, w)$, the set of all finite paths from level $V_0$ to a vertex $w 
\in V_n$ for some $n>0$. This means that every set $E(V_0, w)$ contains one maximal path and one minimal path. For an infinite path  $e= (e_0, e_1,..., e_i,...)$, we say that it  is 
\textit{maximal} if $e_i$ is maximal in the set $r^{-1}(r(e_i))$ for each $i\in 
\N_0$. Similarly, we can define a \textit{minimal path}. We will denote by 
$Y_{max}$ (respectively $Y_{min}$) the set of maximal (respectively minimal) 
infinite paths in $Y_B$. Clearly, $Y_{max}$ and  $Y_{min}$ are Borel sets.

\begin{definition}\label{order} A generalized Bratteli diagram $B=(V,E)$ together with a partial order $\geq$ on $E$ is called a \textit{generalized  
ordered Bratteli diagram} and denoted by $B=(V,E,\geq)$. A generalized stationary Bratteli diagram is called stationary ordered if the linear order
on $r^{-1}(v)$ is the same for all vertices $v$ independent of the level.
\end{definition}

\begin{definition}\label{VM} For a generalized ordered  Bratteli diagram, we define a Borel transformation 
$$
\varphi : Y_B \setminus Y_{max} \rightarrow Y_B \setminus Y_{min}
$$ 
as follows:  given $x = (x_0, x_1,...)\in Y_B\setminus Y_{max}$, let $m$ be the smallest number such that $x_m$ is not maximal. Let $g_m$ be the successor of $x_m$ in the set $r^{-1}(r(x_m))$. Then we define $\varphi(x)= (g_0, g_1,...,g_{m-1},g_m,x_{m+1},...)$ where $(g_0, g_1,..., g_{m-1})$ is the minimal path in $E(V_0, r(g_{m-1}))$. It is not difficult to check that the map $\varphi$  is a Borel bijection. If we can extend $\varphi$ bijectively to the entire path space $Y_B$, then we call the Borel transformation $\varphi: Y_B  \rightarrow Y_B$ a \textit{Vershik map}, and the Borel dynamical system $(Y_B,\varphi)$ is called a generalized \textit{Bratteli-Vershik} system.
\end{definition}

\begin{remark} The Vershik map $\varphi$ admits a Borel extension to $Y_B$ 
if and only if $|Y_{max}|  = |Y_{min}|$. This relation includes the cases when 
the sets $Y_{max}$ and  $Y_{min}$ are finite, countable, and uncountable.
\end{remark}

\begin{lemma} Let $\mu$ be a finite or sigma-finite measure on the path space 
$Y_B$ of a generalized ordered Bratteli diagram, and $\varphi$ the Vershik map. If $\mu(Y_{max}) =
\mu(Y_{min}) =0$, then $\mu$ is tail invariant if and only if $\mu$ is 
$\varphi$-invariant.

\end{lemma}

The following result shows that for any aperiodic Borel automorphism of a standard Borel space there exists a generalized Bratteli-Vershik model.

\begin{theorem}[\cite{BezuglyiDooleyKwiatkowski_2006}]
\label{thm borel-bratteli} Let $T$ be an aperiodic Borel
 automorphism acting
on a standard Borel space $(X, \B)$. Then there exists a 
generalized ordered Bratteli diagram $B=(V,E,\geq)$ and a Vershik automorphism 
$\varphi : X_B \to X_B$ such that $(X, T)$ is Borel isomorphic to 
$(X_B,\varphi)$.
\end{theorem}

\begin{example}\label{read} \textit{Substitution read on a stationary Bratteli diagram}: Let $B=(V,E,\geq)$ be a generalized  stationary ordered Bratteli diagram, and let $\A$ be a countably infinite alphabet. For each $n \in \N_0$, we label the set of vertices $V_n$ by letters in $\A$. We denote by $v(a,n)$ the vertex corresponding to letter $a \in \A$ on level $n$. Thus, for all $n \in \N_0$, we can write $V_n = \{v(a,n); a \in \A\}$. For a fixed $n > 0$ and $a \in V_n$, we consider the finite set $r^{-1}(v(a,n))$, i.e., the set of all edges with range $v(a,n)$. This is an ordered finite set. We use the order to write 
$$r^{-1}(v(a,n)) = (e_1,e_2,...,e_m)
$$ where $e_1 < e_2 <... <e_m$.  Note that $B$ is a stationary diagram, hence the above set is independent of $n$. Let $a_i \in \A$ be such that $s(e_i)$ is labeled by $a_i$ for $i \in \{1,2,...,m\}$. It defines the ordered set 
 $(a_1,a_2,...,a_m)$ of these labels.  This ordered set may contain possible repetitions. Then the  map $a \mapsto a_1a_2 \cdots a_m$ from $\A$ to $\A^+$ defines a substitution on $\A$ which is called the \textit{substitution read on} 
 a stationary ordered generalized Bratteli diagram $B=(V,E,\geq)$. \end{example} 

\begin{remark}\label{paths} We provide here an alternative description for infinite paths in a generalized Bratteli diagram $B = (V, E, \geq)$
which will be used below. We set $h(0, v) =1$ for all $v \in V_0$, and  for 
$v \in V_1$, define $h(1, v) = |r^{-1}(v)|$. 
For $v \in V_n$, $n> 1$,  we define by induction
\begin{equation}\label{height}
    h(n, v) = \underset{w \in s(r^{-1}(v))}{\sum} |E(w, v)| h(n-1, w).
\end{equation} 
Let $v \in V_n $.  All finite paths 
in the set $E(V_0, v)$ are comparable under the lexicographical order described above. This allows us to enumerate the elements of $E(V_0, v)$ from $0$ to $h(n, v) - 1$. In this enumeration, we allot $0$ to the minimal path and $h(n, v) - 1$ to the maximal path. 

Let $y = (e_0,e_1 ... )$ be an infinite path in $Y_B$. Consider a sequence $(Q_n)_{n \in \N}$ of growing finite paths defined by $y$, 
$$Q_n = (e_0, . . . , e_{n-1}) \in E(V_0, r(e_{n-1})).
$$ 
This allows us to identify every $Q_n$ with a pair $(i_n, v_n)$ where $v_n = r(e_{n-1})$ and $i_n \in [0, h(n, v_n) - 1]$ is the order given to $Q_n$ in the set $E(V_0, v_n)$. Thus, for every $y = (e_0,e_1...) \in Y_B$, we can obtain a unique identifier in the form of an infinite sequence $(i_n, v_n)$ where $v_n = r (e_{n-1})$ and $i_n \in [0, h(n, v_n) - 1]$, $n \in \N$.

\end{remark}

\subsection{Stationary generalized Bratteli-Vershik model for substitutions on infinite alphabet}\label{BBD1} 
In this section, we provide an algorithm to construct a generalized stationary  Bratteli-Vershik model for a bounded size left determined substitution $\sigma$ on a countably infinite alphabet $\A$. Let $(X_\sigma,T)$ be the subshift associated with $\sigma$. We use Corollary \ref{CS 2} to obtain a nested sequence of complete sections for $(X_\sigma,T)$ 
$$ 
X_\sigma = A_0 \supset A_1 \supset A_2 \supset A_3....\,\,\,. 
$$ 
Here  the sets $A_n$ are defined as follows: 
$$ 
A_n = \bigsqcup_{a \in \A}  [\sigma^n (a)], \quad n \geq 0.
$$ 
For $n \in \N$, using (\ref{tower 2}), 
we obtain the partition $\xi_n$ of $X_\sigma$:  
\begin{equation}\label{eq 6.3.1}
    \xi_n = \underset{a \in \A}{\bigsqcup} \xi_n(a),
\end{equation} 
where
\begin{equation}\label{eq 6.3.2}
  \xi_n(a) = \{T^i [\sigma^n (a)] : i = 0,..., (h^n_{a} -1)\},\,\,\,\, a \in \A\,, \,\,\, h_a = |\sigma(a)|.
\end{equation} We also define
$$
\xi_0 = \underset{a \in \A}{\bigsqcup} [a].
$$ 
Set $V_0 := \{[a]\}_{a \in \A}$. Each element in the set $V_0$  corresponds to 
a vertex at level $0$ of the generalized Bratteli diagram. 
We  denote by 
$v_{a,0}$ the vertex in $V_0$ corresponding to the cylinder set $[a], a \in \A$. 
 Similarly, we denote by
$$
V_1 := \{[\sigma (a)]\}_{a \in \A}.
$$ 
The set $V_1$ is identified with cylinder sets of the form $[\sigma (a)]
$ for $a \in \A$.  The set $[\sigma(a)]$ corresponds to a vertex $v_{a,1} \in 
V_1$. Note that  $[\sigma (a)]$ is the base of the tower $\xi_1(a)$.  Then relation (\ref{eq 6.3.2}) (for $n = 1$) is represented as
\begin{equation*}\label{eq 6.3.3}
    \xi_1(v_{a,1}) : = \xi_1(a) = \{T^i [\sigma (a)] : i = 0,..., (h_a -1)\},\,\,\,\, v_{a,1} \in V_1,
\end{equation*} and  then (\ref{eq 6.3.1}) (for $n = 1$) as
\begin{equation*}\label{eq 6.3.4}
    \xi_1 = \underset{v \in V_1}{\bigsqcup} \xi_1(v). 
\end{equation*} 
We apply this construction for every $A_n$ in the sequence above to obtain the corresponding partition $\xi_n$ of towers $\xi_n(v)$ of height $h^n_{v} -1$. Note that Remark \ref{refine 2} implies that each $\xi_{n+1}$ refines $\xi_n$ and $\bigcup_{n \in \N} \xi_n$ generates the $\sigma$-algebra $\B$ of Borel sets in $X_\sigma$. With the help of sequence of partitions $(\xi_n)_{n\in \N_0}$ we can start constructing the generalized Bratteli diagram $B=(V,E)$. 

Recall, we denoted by $v_{b,0}$ a vertex in $V_0$ corresponding to cylinder set $[b]$ for $b \in \A$. To define $E_1$, we take $v_{a,1} \in V_1$ and corresponding cylinder set $[\sigma (a)]$. Here we are working with a nested sequence of K-R towers. Thus every cylinder set of the form $[\sigma (a)]$ lies in a cylinder set of the form $[b_1]$ for some $b_1 \in \A$. 

Since the cylinder set $[\sigma (a)]$ is base of a tower $(\xi_1(a))$ of height $h_{a,1} -1$, each level of this tower lies in a cylinder set of the form $[b_i]$ for some $b_i \in \A$ (which were used to construct the vertex of the $0$th level). Let $\{[b_1],[b_2]...,[b_{h_a}]\}$, be the list of these cylinder sets from bottom to top. We connect $v_{a,1} \in V_1$ to vertices $v_{b_k, 0}$ in $V_0$, for $k \in \{1,..,h_a\}$ with an edge. We do this for each $v \in V_1$, to define $E_1$. To introduce a linear order on $r^{-1} v $ for $v \in V_1$, we enumerate them from $1$ to $h_v$ in the order they appear in the tower going bottom to top.

We repeat this procedure for every $n \in \N_0$ to obtain the corresponding $V_n$ and partially ordered $E_n$. This gives us an ordered generalized Bratteli diagram $B(V,E,\geq)$, with $V=\bigsqcup_{i\in\N_0} V_i$ and $E=\bigsqcup_{i\in\N_0} E_i$. It is easy to check that this diagram satisfies the Definition \ref{def GBD}. The fact that $\big|\underset{n\in \N_0}{\bigcap} A_n\big| = \aleph_0$ implies that the maximal and minimal paths in $B$ are countably infinite. Note that every infinite path $y \in Y_B$ is completely determined by the infinite sequence ${(i_n, v_n)}_{n}$,
$v_n \in V_n$, $0 \leq i_n \leq h(n, v_n)-1$ such that $T^{i_{n +1}} ([\sigma^{n +1}(a_{n+1})]) \subset T^{i_n} ([\sigma^{n}(a_{n})])$,  for every $n \in \N_0$, where the cylinder set $[\sigma^{n}(a_{n})]$ corresponds to $v_n \in V_n$ and so on. 

We claim that the generalized Bratteli diagram $B = (V, E, \geq)$ constructed above is a stationary diagram and $\sigma$ is in fact the substitution read on $B = (V, E, \geq)$. To see this fix $n \in \N$ and consider the vertex sets $V_{n-1}$ and $V_n$ at level $(n-1)$ and $n$ respectively. In the construction above, the partition $\xi_n$ is used to obtain the edge set $E_n$. Recall that we identified vertices at level $n$ with sets of the form $[\sigma^n (a)]$ for $a\in \A$. These sets are the base of the tower $\xi_n(a)$ (see (\ref{eq 6.3.1})). For a vertex $v \in V_n$, write $v = V(a_v, n)$ (see Remark \ref{read}) to illustrate that $v$ is a vertex on level $n$ corresponding to letter $a_v \in \A$.  

There exists $a_1 \in \A$ such that we have the following inclusion, $[\sigma^n (a_v)] \subset [\sigma^{n-1} (a_1)]$. Here $[\sigma^{n-1} (a_1)]$ is the base of the tower $\xi_{n-1}(a_1)$ in partition $\xi_{n-1}$. Now we consider the set of edges with range $v$; i.e. the set $r^{-1}(v)$. By the above construction, this set is fully ordered. We list its elements in increasing order $(e_1,...,e_s)$. For $i\in \{1,..,s\}$, let $a_i$ be the label of the vertex (see Remark \ref{read}) on level $V_{n-1}$ which is the source of edge $e_i$. We make an ordered list of these labels $(a_1,..,a_s)$ (i.e. same order as $e_i$'s). 

By the above construction this means that as we move from bottom to top in the $T$-tower $\xi_n(a_v)$ (with base $[\sigma^n (a_v)]$) it intersects with $T$-towers $\{\xi_1(a_1),...,\xi_i(a_s)\}$ in the given order (i.e. $1$ to $s$). This in turn implies $$
\sigma^n(a_v) = \sigma^{n-1}a_1....\sigma^{n-1}a_s.
$$ Hence we get $\sigma(a_v) = a_1....a_s$. Note that we will get the same expression for $\sigma(a_v)$ using any level $n \in \N$, hence the diagram $B = (V, E, \geq)$ is stationary. Moreover, the substitution we recovered above (from the diagram) matches the substitution $\sigma$. This implies that $\sigma$ is the substitution read on $B = (V, E, \geq)$. 

Now we are ready to prove Theorem \ref{Main_1} which states that the subshift $(X_{\sigma}, T)$ corresponding to a bounded size left determined substitution $\sigma$ on an infinite alphabet can be realized as a Vershik transformation acting on the path space of an ordered stationary generalized Bratteli diagram. We recall the statement below.

\begin{theorem}\label{isom1} Let $\sigma $ be a bounded size left determined substitution on a countably infinite alphabet and $(X_\sigma,T)$ be the corresponding subshift. Then there exists a stationary ordered generalized Bratteli diagram $B = (V, E, \geq)$ and a Vershik map $\varphi : Y_B \rightarrow Y_B$ such that $(X_\sigma, T)$ is isomorphic to $(Y_B, \varphi)$.

\end{theorem}

\noindent\textit{Proof}. We use Corollary \ref{CS 2}, to obtain a nested sequence of complete sections $(A_n)_{n\in \N_0}$ corresponding to $(X_\sigma,T)$. As described above for $n \in \N_0$, let $$\xi_n = \xi_n(v), v\in V_n$$ be the partition of $X_\sigma$ corresponding to $A_n = \underset{a \in \A}{\bigsqcup}  [\sigma^n(a)]$. As before by identifying $V_n \ni v_{a,n} \sim [\sigma^n (a)]$. For convenience, we write $A_n(v) : = [\sigma^n (a)] $. Thus in general for $v \in V_n$, $n \in \N_0$, we write

$$
\xi_n(v) = \{T^j A_n(v) : j  = 0,...,h(n,v) - 1\}.$$ Here $h(n,v) = h^n(a)$ (and $[\sigma^n(a)]$ is the set corresponding to $v \in V_n$). Remark \ref{refine 2}, implies that atoms of $(\xi_n, n\in \N_0)$ generate the Borel structure of $X_{\sigma}$ and are clopen. In other words, sets of the form  $T^j(A_n(v)),\,\, v\in V_n,\,j = 0,...,h(n,v)-1, n\in \N_0$, generate the Borel structure of $X_{\sigma}$.

For a fixed $n \in \N_0$ and $\epsilon >0$, we can cut each tower $\xi_n(v)$ into disjoint clopen towers of the same height with the diameter of each new tower less than $\epsilon$. Thus we can assume, 
\begin{equation}\label{dia1}
    \underset{0\leq j < h(n,v), v\in V_n}{\textrm{sup}} \,\,[\textrm{diam}\, T^j (A_n(v))]\rightarrow 0,\,\, n\rightarrow \infty.
\end{equation} Here we use the metric defined in (\ref{metric}). 

We apply the construction discussed above to the sets $(A_n)$ and $T$ to obtain a stationary ordered generalized Bratteli diagram $B = (V, E, \geq)$. We define a map $\varphi$ (Vershik map) on the path space $Y_B$ as follows: On $Y_B \setminus Y_{max}$, the $\varphi$ is defined using the algorithm in Definition \ref{VM}. Thus $\varphi$ is a homeomorphism from $Y_B \setminus Y_{max}$ to $Y_B \setminus Y_{min}$. By construction, the cardinality of $Y_{max}$ is equal to the cardinality of $Y_{min}$ (both are countably infinite). Hence we extend $\varphi$ to entire $Y_B$ by defining a bijective Borel transformation mapping the set of maximal paths to the set of minimal paths; i.e. $\varphi(Y_{max}) = Y_{min}$. To do so, consider a maximal path $y \in Y_{max}$. Note that $y$ determines a unique sequence ${(i_n, v_n)}_n$, where $v_n \in V_n$ and $ i_n = h(n, v_n) - 1 $ (since $y$ is a maximal path) for every $n \in \N_0$. More over $T^{i_{n +1}} (A_{n +1}(v_{n+1})) \subset T^{i_n} (A_n(v_{n})),$
for every $n \in \N_0$. Thus $y\in Y_{max}$ corresponds to a unique point $x \in X_{\sigma}$ where $\{x\} = \underset{n}{\bigcap}\,T^{i_n} A_{n}(v_n)$. Uniqueness of $x$ follows from (\ref{dia1}).

Let $\tilde{x} = T (x)$; since $T$ takes points at the top of a KR tower to points at the bottom of the tower, $\tilde{x}$ lies in the intersections of the base of the nested sequence of Kakutani-Rokhlin towers. Again $\tilde{x}$ determines a unique sequence ${(i_n, v_n)}_n$, where $v_n \in V_n$ and $ i_n = 0$ (since $x$ lies in bottom of the towers) for every $n \in \N_0$ and $\{\tilde{x}\} = \underset{n}{\bigcap}\, A_{n}(v_n)$. Let $\tilde{y}$ be the unique infinite path in $Y_B$ corresponding to $\tilde{x}$ then by construction $\tilde{y} \in Y_{min}$. We define $\varphi(y) = \tilde{y}$. This gives a bijective Borel transformation mapping $Y_{max}$ to $Y_{min}$. The bijection follows from
the bijection of $T$. This together with the definition of $\varphi: Y_B\setminus Y_{max} \rightarrow Y_B\setminus Y_{min}$ (see Definition \ref{VM}) gives us a Borel automorphism $\varphi: Y_B \rightarrow Y_B$ which we call the Vershik map.

Now we show that $(X_\sigma,T)$ is Borel isomorphic to $(Y_B,\varphi)$. To do this we define a map $f: X_{\sigma} \rightarrow Y_B$  as follows: For $x \in X_{\sigma}$, choose the unique sequence ${(i_n, v_n)}_n$, $0 \leq i_n \leq h(n, v_n) - 1 $, $v_n \in V_n$, such that \begin{equation}\label{inc1}
    T^{i_{n +1}} (A_{n +1}(v_{n+1})) \subset T^{i_n} (A_{n}(v_{n})),
\end{equation} for every $n \in \N_0$ and $\{x\} = \underset{n}{\bigcap}\,T^{i_n} A_{n}(v_n)$. As mentioned in Remark \ref{paths}, such a sequence defines a unique infinite path $y \in Y_B$. We set $f(x) = y$. Observe that $f$ is a continuous injective map. From the construction of $B = (V,E, \geq)$ and definition of Vershik transformation $\varphi$, acting on $Y_B$, we get that $f(T x) = \varphi(f x)$, $x \in X_{\sigma}$. We claim that $f$ is surjective. To see this, let $y \in Y_B$ be an infinite path. Then $y$ defines an infinite sequence ${(i_n, v_n)}_n$, $0 \leq i_n \leq h(n, v_n) - 1 $, $v_n \in V_n$, (see Remark \ref{paths}). Thus by (\ref{dia1}) and (\ref{inc1}), we obtain a single point $x$ such that $f(x) = y$.   \hfill{$\square$} 

In Appendix \ref{BBD2}, we provide an alternate construction of a generalized non-stationary  Bratteli-Vershik model for a subshift associated with left determined substitution $\sigma$ on a countably infinite alphabet $\A$. We do not assume that $\sigma$ is bounded size, and the generalized B-V model, in this case, is not necessarily stationary (See Theorem \ref{isom2}).

\section{Invariant measure for infinite substitution via Bratteli diagrams}\label{sect inv} 

For a stationary generalized Bratteli diagram with an irreducible, aperiodic, and recurrent countable incidence matrix (see Definition \ref{def PF}), there exists a direct method for finding  a tail invariant measure  on the path space
of the diagram. 
In the previous section, we constructed an isomorphism between the subshift (associated with bounded size left determined substitution on a countably infinite alphabet) and the Vershik map on path space of the corresponding stationary generalized Bratteli diagram. Therefore this method will also provide a formula for a shift-invariant measure when the subshift is associated with a bounded size left determined substitution on countably infinite alphabets.

\subsection{Tail invariant measure for stationary Bratteli diagrams}\label{tail_inv} Recall that a generalized Bratteli diagram $B = B(F_n) $ is called
\textit{stationary} if, for every $n \in \N_0$, the incidence matrices $F_n = F$. Thus, the structure of the diagram remains stationary, and the set of edges  $E_n$ between $V_n$ and $V_{n+1}$ is independent of $n \in \N_0$. 
In the rest of the paper,
 we will identify the set of vertices $V_n$ with the integers $\Z$ for all $n$. 
 
 We discuss below some definitions and results from the theory of 
 non-negative infinite matrices. We refer the reader to Chapter $7$ of \cite{Kitchens1998} for a detailed discussion of the Perron-Frobenius 
 theory of such matrices. 

\begin{definition}\label{def PF} \noindent $(a)$ We call a matrix $F = (f_{ij})$  \textit{infinite} if its rows and columns are indexed by the same countably 
infinite set. Assuming that all matrices $F^n$ are defined (i.e., they have finite entries),  we denote the entries of $F^n$ by $f^{(n)}_{ij} >0$.

\vspace{2mm}

\noindent $(b)$ An infinite non-negative matrix $F = (f_{ij})$, $i,j \in \Z$, is called \textit{irreducible} if, for all $i, j \in \Z$, there exists some $n \in \N_0$ 
such that $f^{(n)}_{ij} >0$. A stationary Bratteli diagram $B =B(F)$ is called an \textit{irreducible} diagram if its incidence matrix $F$ is irreducible. 

\vspace{2mm}

\noindent $(c)$ An irreducible matrix $F$ is said to have period $p$ if, for all vertices $i \in \Z$, 
$$
p = \gcd \{ t : f^{(t)}_{ii} >0\}.
$$ 
An irreducible matrix $F$ is called \textit{aperiodic} if $p = 1$.

\vspace{2mm}

\noindent $(d)$ For an irreducible, aperiodic, non-negative, infinite matrix $F$ the limit defined by 
$$
\lambda = \lim_{n\to \infty} (f^{(n)}_{ii})^{\frac{1}{n}}
$$ 
exists and is independent of the choice of $i$. The value of the limit $\lambda$ is called the \textit{Perron-Frobenius eigenvalue} of $F$. We will 
consider only the case of finite Perron-Frobenius eigenvalue $\lambda$.

\vspace{2mm}

\noindent $(e)$ It is said that an irreducible, aperiodic, non-negative, infinite matrix $F$ is 
\textit{transient} if 
$$
\sum_{n\geq 1} f^{(n)}_{ij} \lambda^{-n} < \infty.
$$

\vspace{2mm}

\noindent $(f)$ It is said that an irreducible, aperiodic, non-negative, infinite matrix $F$ is 
\textit{recurrent} if 
$$
\sum_{n\geq 1} f^{(n)}_{ij} \lambda^{-n} = \infty.
$$

\vspace{2mm}

\noindent $(g)$ For an irreducible, aperiodic, recurrent, non-negative, infinite matrix $F$, let $t_{ij}(1) = f_{ij}$, and for $n> 1$,
$$
t_{ij}(n+1) =\sum_{k \neq i} t_{ik}(n) f_{kj}.
$$ 
It is said the matrix $F$ to be \textit{null-recurrent} if
$$
\sum_{n\geq 1} n t_{ii}(n) \lambda^{-n} < \infty;
$$
otherwise,  $F$ is called \textit{positive recurrent}.
\end{definition} The theorem below is a generalization of the classic Perron-Frobenius theory to non-negative infinite matrices. We refer the reader to \cite{Kitchens1998} (Theorem $7.1.3$) for a detailed proof.

\begin{theorem}[Generalized Perron-Frobenius theorem] \label{PF1}
Let  $F$ be an irreducible, aperiodic, recurrent, non-negative, infinite matrix. Then there exists a Perron-Frobenius
 eigenvalue 
 $$\lambda = 
 \lim_{n\to \infty} (f^{(n)}_{ij})^{\frac{1}{n}} >0
 $$ 
(assumed  to be finite) such that:

\vspace{2mm}

\noindent $(a)$ there exist strictly positive left $\ell$ and right $r$  eigenvectors for $\lambda$;

\vspace{2mm}

\noindent $(b)$ both $\ell$ and $r$ are unique  up to multiplication by constants;

\vspace{2mm}

\noindent $(c)$ $F$ is positive recurrent if and  only if $\ell \cdot r = \sum_i \ell_i r_i <\infty$;

\vspace{2mm}

\noindent $(d)$ if $F$ is null-recurrent, then $\displaystyle{\lim_{n\to \infty} F^n\lambda^{-n}=0}$;

\vspace{2mm}

\noindent $(e)$ if $F$ is positive  recurrent, then $\displaystyle{\lim_{n\to \infty} F^n \lambda^{-n}= r \ell}$ (normalized so that $\ell r =\textbf{1}$).
\end{theorem}

\begin{remark}\label{aper} The aperiodicity assumption can be removed from the statement of Theorem \ref{PF1} (see \cite{Kitchens1998} Lemma $7.1.37$ and Lemma $7.1.38$). Hence all the results in this section are true without assuming aperiodicity. But to be consistent with the literature we have stated the results keeping the aperiodicity assumption. 

\end{remark}

The following result (proved in \cite{Bezuglyi_Jorgensen_2021}) provides an explicit formula for a tail invariant measure on the path space of a stationary generalized Bratteli diagram. Since this result is crucial for the rest of the paper, we reproduce the proof in Appendix \ref{proof}.

\begin{theorem} [Theorem 2.20, \cite{Bezuglyi_Jorgensen_2021}] \label{inv1}  Let $B = B(F)$ be a stationary generalized Bratteli diagram such that the incidence matrix $F$ is irreducible, aperiodic and recurrent. Then

\noindent $(1)$ there exists a tail invariant measure $\mu$ on the path space $Y_B$,

\noindent $(2)$  the measure $\mu$ is finite if and only if the left 
Perron-Frobenius eigenvector $\ell = (\ell_v)$ has the property $\sum_{v} \ell_v < \infty$.
\end{theorem}

A special case of a recurrent matrix is positive recurrent. If we assume that the incidence matrix $F$ is positive recurrent and has an equal row sum 
property, then the tail invariant measure defined above is always finite.

\begin{corollary}\label{inv2} Let $B = B(F)$ be a stationary generalized Bratteli diagram such that the incidence matrix $F$ is irreducible, aperiodic, and positive recurrent. Assume that $F$ has equal row sum property, i.e. for every $v, v' \in V$,
$$
\sum_{w \in V} f_{v, w} = \sum_{w \in V} f_{v', w}.
$$ 
Then the tail invariant measure $\mu$ (defined in (\ref{eq inv meas left})) on the path-space $Y_B$ is finite.
\end{corollary}

\begin{proof} Since $F$ is positive recurrent, we have 
$$\ell \cdot r = \sum_{v\in V} \ell_v r_v < \infty,
$$ where $r$ denotes the right eigenvector corresponding to $\lambda$. Moreover, the equal row sum property implies that we can take  $r = 
(...1,1, 1,...)$. Hence 
$$
\mu(Y_B)  = \sum_{i \in \Z} \ell_i < \infty.
$$ 
\end{proof}

\subsection{From tail invariant measure to shift-invariant measure} In Subsection \ref{tail_inv}, we have provided an explicit formula for a tail invariant (or invariant under the Vershik map) measure on a stationary 
generalized Bratteli diagram with an irreducible, aperiodic, and recurrent incidence matrix. In this subsection, we will use the isomorphism between the 
subshift (associated with a  bounded size left determined substitution 
on a countably infinite alphabet) and the Vershik map on path space of the corresponding stationary generalized ordered Bratteli diagram (see Theorem \ref{isom1}) to obtain a shift-invariant measure. 

Note that if $\sigma$ is the substitution read on a stationary generalized Bratteli diagram $B = B(F)$ (see Remark \ref{read}) then the infinite substitution matrix of $\sigma$ (see Definition \ref{inf sub}) is same as the incidence matrix $F$. Now we are ready to prove Theorem \ref{Main_2}. We recall the statement below.

\begin{theorem}\label{inv3} Let $\sigma $ be a bounded size left determined substitution on a countably infinite alphabet $\A$ and $(X_\sigma, T)$ be the corresponding subshift. Assume that the countably infinite substitution matrix $M$ is irreducible, aperiodic, and recurrent. Then

\vspace{2mm}

\noindent $(1)$  there exists a  shift-invariant measure $\nu$ on $X_\sigma$;

\vspace{2mm}

\noindent $(2)$ the measure $\nu$ is finite if and only if the Perron-Frobenius
 left eigenvector  $\ell = (\ell_i)_{i \in \Z}$ of the matrix $M$ has the property $\underset{i \in \Z}{\sum} \ell_i < \infty$. . 

\end{theorem}

\begin{proof} By Theorem \ref{isom1}, there exists a stationary generalized Bratteli diagram $B(F)$ such that $(X_\sigma,T)$ is isomorphic to the dynamical system $(Y_B,\varphi)$. Let  $f: (X_{\sigma}, T_\sigma) \rightarrow (Y_B, \varphi)$ implement this isomorphism. Also, by construction,  the incidence matrix $F$ of the diagram coincides with the substitution matrix $M$ (see Section \ref{BBD1}). Since  $M$ is irreducible and recurrent, there exists a 
tail invariant measure $\mu$ on the path space $Y_B$ (Theorem \ref{inv1}). 
We define the measure $\nu$ on $X_{\sigma}$ by setting  
\begin{equation}\label{mes}
    \nu (A) = \mu \circ f (A)
\end{equation} 
where  $A \subset X_{\sigma}$ is a Borel set. 
The rest of the proof is obvious. 
\end{proof}

\begin{corollary}\label{inv4} Let $\sigma $ be a constant length left determined substitution on a countably infinite alphabet such that the substitution matrix $M$ is irreducible, aperiodic, and positive recurrent. Then the shift-invariant measure $\nu$ on $X_\sigma$ defined in (\ref{mes}) is finite.
\end{corollary}

\begin{remark}\label{erg} In a sequel to this paper we show that if $B = B(F)$ be a stationary generalized Bratteli diagram with irreducible, aperiodic and \textit{positive} recurrent incidence matrix $F$ then the shift-invariant \textit{finite} measure (defined in (\ref{eq inv meas left})) on the path-space $Y_B$ is \textit{ergodic}. The proof involves a direct application of the pointwise ergodic theorem. Since we prove it more generality in the upcoming sequel we have not mentioned the proof here. 

As a consequence for a bounded size, left determined substitution $\sigma$ (on a countably infinite alphabet) with irreducible, aperiodic, and \textit{positive} recurrent substitution matrix the shift-invariant \textit{finite} measure $\nu$ on $X_\sigma$ defined in (\ref{mes}) is \textit{ergodic}.
 
\end{remark}

\section{Examples}\label{Sec ex} In this section we provide some examples of bounded size, left determined substitution on countably infinite alphabets for which stationary generalized Bratteli-Vershik diagram can be constructed using methods in Section \ref{Sec BBD}.

\begin{example} \textit{One step forward, two step back substitution on $\Z$}. Define $\sigma$ by $$
-1 \mapsto -2 -1 \,\,\,0\,\,\,;\,\,\, 0 \mapsto -1 \,\,\, 0 \,\,1
$$
$$n \mapsto (n-1)(n+1)(n+1)\,\,;\,\,n \leq -2$$
$$n \mapsto (n-1)(n-1)(n+1)\,\,;\,\,n \geq 1 $$ This substitution is bounded size and left determined, hence by Theorem \ref{isom1} there exists a stationary generalized Bratteli-Vershik model for the associated subshift. Figure \ref{fig:KRT} shows part of $V_n$, $E_n$ and $V_{n+1}$ for any $n \in \N_0$ of the stationary generalized Bratteli-Vershik model.

\begin{figure}[!htb]
    \centering
    \captionsetup{justification=centering}
    \includegraphics[width=100mm]{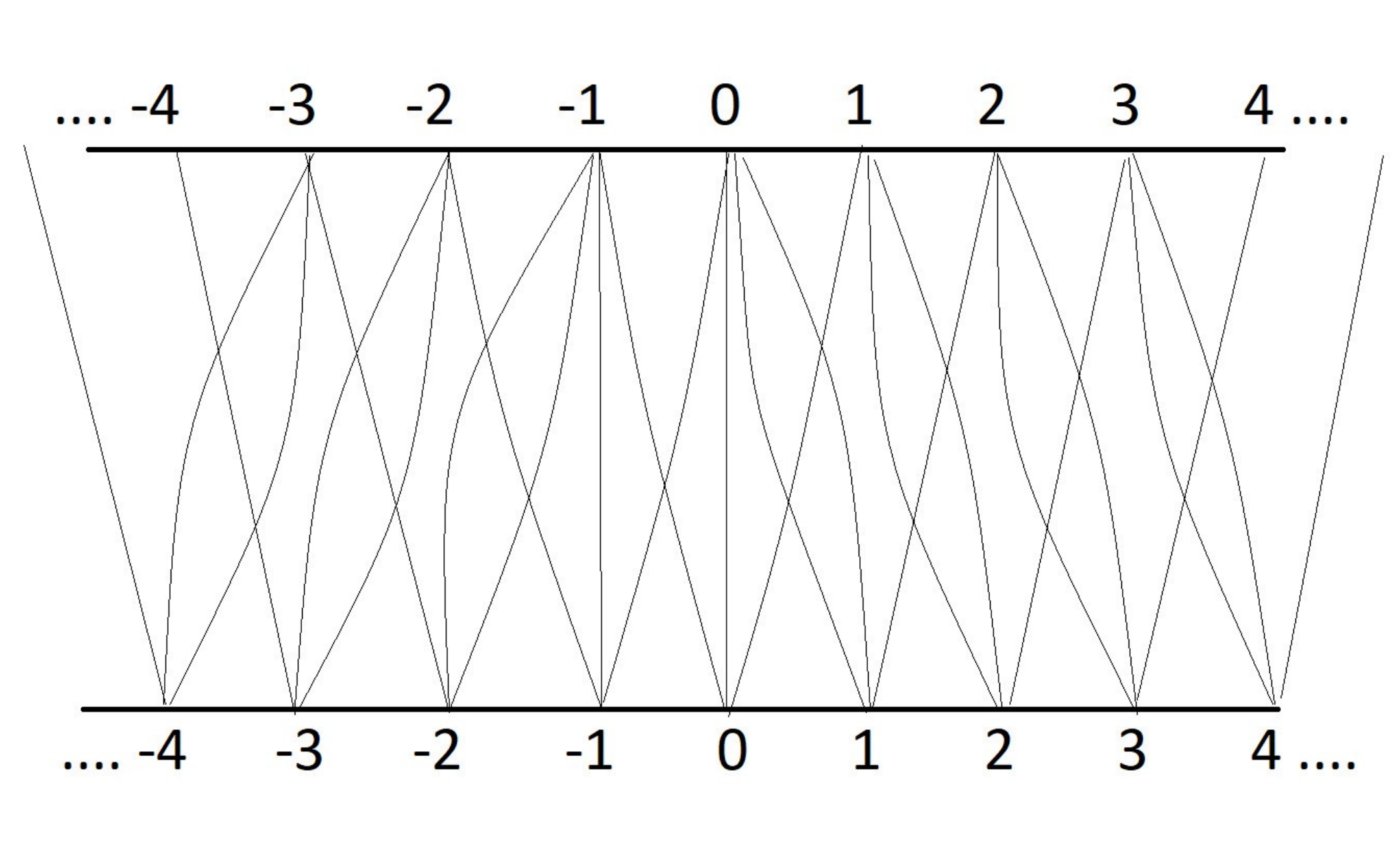}
    \caption{Generalized B-V model for Example 8.1.}
    \label{fig:KRT}
\end{figure} Since the substitution matrix is irreducible, aperiodic, and recurrent, we can apply Theorem \ref{inv3} to obtain an expression for a shift-invariant measure. Manual calculation yields P.F value $\lambda = 3$ and the left P.F eigenvector is given by$$\ell = (...\frac{1}{2^4},\frac{1}{2^3},\frac{1}{2^2},\frac{1}{2},1,1,\frac{1}{2},\frac{1}{2^2},\frac{1}{2^3}.....).$$ The right P.F eigenvector has all entries equal to $1$. By normalizing the left P.F eigenvector such that $\sum_{v} \ell_v = 1$, we obtain a \textit{probability} shift-invariant measure $\nu$ (defined in (\ref{mes})) on the associated subshift.

\end{example}

\begin{example} \textit{The squared one step forward, two steps backward, substitution on $2\Z$}. Define $\sigma$ by $$
-2 \mapsto -4 -4\,-2-2-2\,\,-4-2-2\,\,0;\,\,\, 0 \mapsto -200022002
$$
$$
n \mapsto (n - 2)(n - 2)n(n - 2)(n - 2)nnn(n + 2) \,\, \textrm{for}\,\, n \in 2\Z \setminus \{0,-2\}
$$ This substitution is again bounded size and left determined, thus by Theorem \ref{isom1} there exists a stationary generalized Bratteli-Vershik model for the associated subshift. Note that the substitution matrix is irreducible, aperiodic, and recurrent, hence we can apply Theorem \ref{inv3} to obtain an expression for a shift-invariant measure. By manual calculations, we get $\lambda = 9$. The left P.F eigenvector is given by$$\ell = (...\frac{1}{2^8},\frac{1}{2^6},\frac{1}{2^4},\frac{1}{2^2},\frac{1}{3},\frac{1}{3},\frac{1}{2^2},\frac{1}{2^4},\frac{1}{2^6},\frac{1}{2^8}.....).$$ The right P.F eigenvector has all entries equal to $1$. By normalizing the left P.F eigenvector such that $\sum_{v} \ell_v = 1$, we obtain a \textit{probability} shift-invariant measure $\nu$ (defined in (\ref{mes})) on the associated subshift.

The next two examples have recurrent substitution matrix with the sum $\sum_{v} \ell_v = \infty$. Hence the shift-invariant measure is sigma-finite.

\end{example} 

\begin{example}
\textit{Random walk on $\Z$}. Define $\sigma$ by $$
n \mapsto (n - 1)(n + 1)
$$  Note that random walk on $\Z$ is a  bounded size, left determined substitution, hence by Theorem \ref{isom1} there exists a stationary generalized Bratteli-Vershik model for the associated subshift. Figure \ref{fig:RW} shows part of $V_n$, $E_n$ and $V_{n+1}$ for any $n \in \N_0$ of the stationary generalized Bratteli-Vershik model.

\begin{figure}[!htb]
    \centering
    \captionsetup{justification=centering}
    \includegraphics[width=100mm]{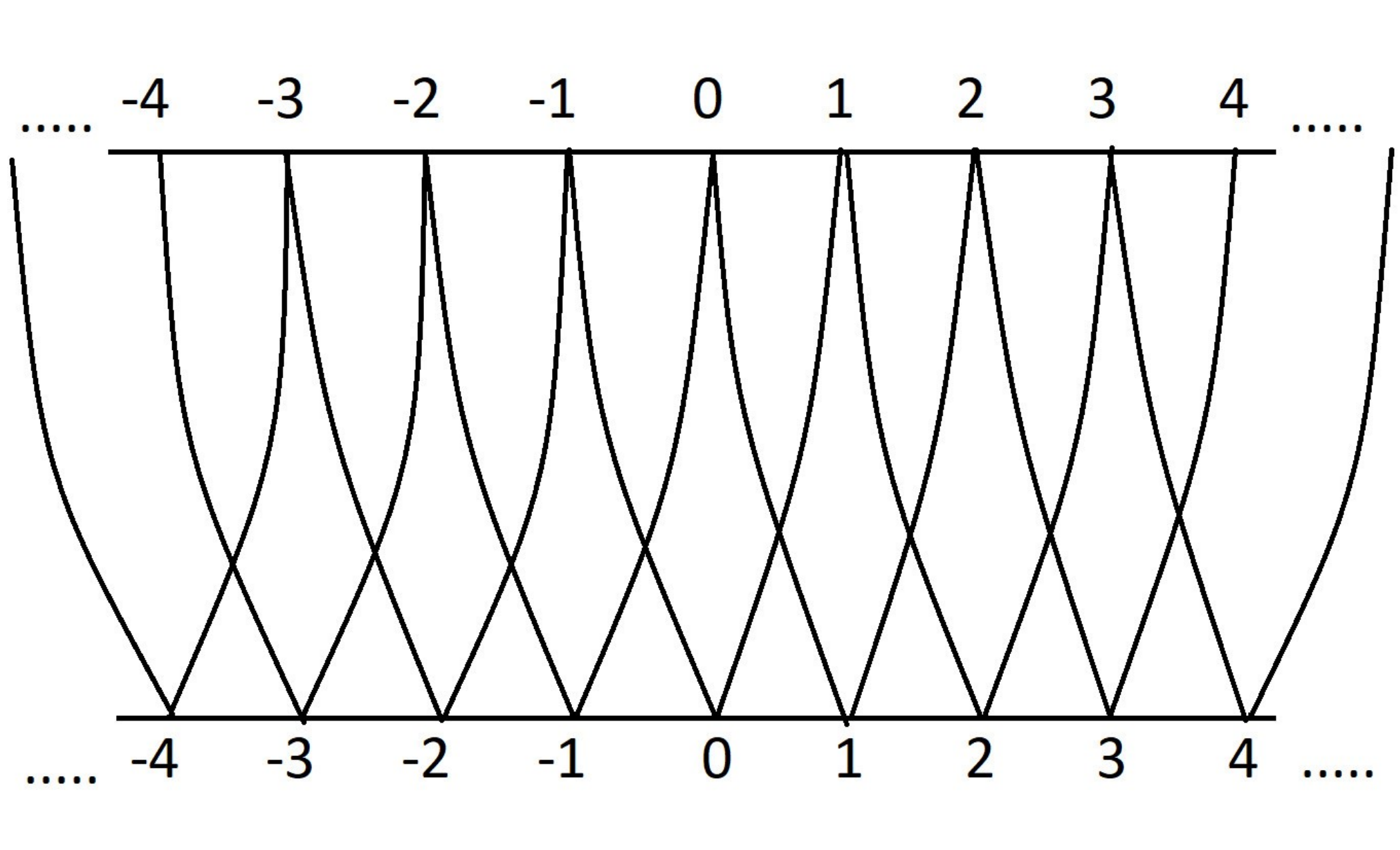}
    \caption{Generalized B-V model of the Random walk on $\Z$.}
    \label{fig:RW}
\end{figure}

The substitution matrix is irreducible, periodic with period 2, and null recurrent. Although the matrix is not aperiodic, we can still study the substitution (see Remark \ref{aper}). Manual calculation yields P.F value $\lambda = 2$ and both the left and right P.F eigenvectors are given by$$r = \ell = (.....\frac{1}{2},\frac{1}{2},\frac{1}{2},\frac{1}{2},.....).$$ Hence by Theorem \ref{inv3}, the shift-invariant measure $\nu$ (defined in (\ref{mes})) on $X_\sigma$ is a sigma-finite, infinite measure. 

\end{example}

\begin{example} \textit{Squared drunken man substitution on $2\Z$}. Define $\sigma$ by $$
n \mapsto (n - 2)nn(n + 2) \,\,; \,\,\,n \in 2 \Z
$$ This substitution is also bounded size and left determined, hence by Theorem \ref{isom1} there exists a stationary generalized Bratteli-Vershik model for the associated subshift. The substitution matrix is irreducible, aperiodic, and recurrent, hence we can apply Theorem \ref{inv3} to obtain the expression for a shift-invariant measure. Manual calculation shows P.-F. value $\lambda = 4$ and both the left and right P.F eigenvectors are given by$$r = \ell = (.....1,1,1,1,1,....).$$ Thus by Theorem \ref{inv3}, the shift-invariant measure $\nu$ (defined in (\ref{mes})) on $X_\sigma$ is a sigma-finite, infinite measure.

\end{example}

\textbf{Acknowledgments} The authors are pleased to thank Olena Karpel for many helpful discussions. This article is a part of the third-named author’s Ph.D. thesis, written under the supervision of the first and second-named authors at the University of Iowa. We thank our colleagues at the University of Iowa, USA, Ben Gurion University of the Negev, Israel, and The Hebrew University of Jerusalem, Israel, for fruitful discussions and support. We thank the referee for their detailed comments on the paper which helped us to improve the exposition.

\appendix
\section{}
\subsection{Some additional versions of the results}\label{App_A}

Theorem \ref{RK_n} is an alternate version of K-R partitions for subshift associated with left determined substitution on countably infinite alphabets. Here we, do not require the assumption of bounded length (compare with Theorem \ref{Rok 1}). 

\begin{theorem}\label{RK_n} Let $\sigma $ be a left determined substitution on a countably infinite alphabet $\A$, and let $(X_{\sigma},T)$ be the corresponding subshift. Then, for any $n \in \N$, there exists a subset $\mathcal{P}_n \subset \mathcal{L}_n$ of words of length $n$, such that $X_\sigma$ can be partitioned into Kakutani-Rokhlin towers given by

\begin{equation} \label{KR2}
    X_{\sigma} = \underset{a_0\cdots a_{n-1} \in \mathcal{P}_n(\sigma)}{\bigsqcup} \,\,\, \underset{k = 0}{\overset{h_{a_0} + ...+h_{a_{n-1}}  - 1}{\bigsqcup}} \, T^k [\sigma (a_0\cdots a_{n-1})]
\end{equation} where $h_{a_k} = |\sigma (a_k)|$ for $k \in \{0,..,n-1\}$. 
\end{theorem}

\begin{proof} The proof for $n =1$ follows from (\ref{tow1}). We will prove the theorem for $n=2$ in two steps. In the first step, we show that the expression 
\begin{equation}\label{exp}
     \underset{a_ia_j \in \mathcal{L}_2(\sigma)}{\bigsqcup} \,\, \underset{k = 0}{\overset{h_{a_i} + h_{a_j}  - 1}{\bigsqcup}} \, T^k [\sigma (a_ia_j)]
\end{equation} 
contains two disjoint copies of $X_{\sigma}$ given by 
\begin{equation} \label{eq Rokh1.11}
    X_{\sigma} = \underset{a_ia_j \in \mathcal{L}_2(\sigma)}{\bigsqcup} \,\,\, \underset{k = 0}{\overset{h_{a_i}  - 1}{\bigsqcup}} \, T^k [\sigma (a_ia_j)]
\end{equation} 
and 
\begin{equation} \label{eq Rokh1.12}
    X_{\sigma} = \underset{a_ia_j \in \mathcal{L}_2(\sigma)}{\bigsqcup} \,\,\, 
  \underset{k = h_{a_i}}{\overset{h_{a_i} + h_{a_j}  - 1}{\bigsqcup}} \, T^k [\sigma 
  (a_ia_j)].
\end{equation} 
In the second step, we provide an algorithm to remove a copy of $X_{\sigma}$ from 
(\ref{exp}), thus proving (\ref{KR2}). The first step is similar to the proof of 
Theorem \ref{Rok 1}. To see that (\ref{eq Rokh1.11}) holds, 
let $x = \{x_i\}_{i \in \Z} \in 
X_\sigma$ and consider the word $w = x_{-\ell}\cdots x_0\cdots x_\ell$ of 
length $2\ell + 1 $ (here $\ell$ is the same as in 
the proof of Theorem \ref{Rok 1}).
Again, using the fact that $\sigma$ is left determined, we obtain a unique 
representation of $w$ given by
\begin{equation}\label{w1_dec}
   w = w_{-p}w_{-p+1}\cdots w_0\cdots w_{s-1}w_s = w_{-p}\sigma (a_{-p+1})\cdots \sigma (a_0)\cdots  \sigma (a_{s-1}) w_s
\end{equation} 
for $a_j \in \A\,,\, j \in \{-p,...,0,...,s\}$ where $p,s$ are positive integers. We  
use the labeling $\{-p,...,0,...,s\}$ such that $x_0 \in \sigma (a_0)$. Also, by 
definition of left determined substitutions, $w_{-p}$ and $w_s$ are either empty words or 
a suffix of $\sigma (a_{-p})$ and a prefix of $\sigma (a_s)$, respectively.

Instead of considering $x_0 \in \sigma(a_0)$, we consider $x_0 \in \sigma(a_0a_1)$ where $a_0a_1 \in \mathcal{L}_2(\sigma)$. Again as in proof of Theorem \ref{Rok 1}, we denote $h_0 : = h_{a_0} = |\sigma (a_0)|$ and set $\sigma (a_0) = b_{1}\cdots b_{h_0}$, where $b_{i} \in \A$ for $i \in \{1,2,... ,h_0\}$. Note that by uniqueness of the decomposition in (\ref{w1_dec}), the 
two-letter word $a_0a_1 \in \mathcal{L}_2(\sigma)$ is uniquely determined. Moreover, $x_0 = b_{j}$ for some $j \in \{1,2,...,h_0\}$. Define
\begin{equation*}
    y = (y_i),\,  y_i = x_{i-(j-1)},\, \mathrm{for}\,\mathrm{all}\,\, i \in \Z.
\end{equation*} Set $j-1 = k$, then $T^{-k} x = y$. Since $y \in [\sigma (a_0a_1)]$, we obtain $x \in T^{k} [\sigma (a_0a_1)] $ for some $k \in \{0,...,h_0 -1\}$. This proves (\ref{eq Rokh1.11}). Since $\sigma$ is left determined, (from an argument similar to the proof of 
Theorem \ref{Rok 1}) it follows that the two unions in (\ref{eq Rokh1.11}) are disjoint. 

To prove (\ref{eq Rokh1.12}),  we consider $x \in \sigma (a_{-1}a_0)$ 
(instead $x \in \sigma (a_0a_1)$). 
Repeating the  argument similar to that above,
 we will show that  (\ref{eq Rokh1.12}) holds. This proves  
that the expression (\ref{exp}) contains two disjoint copies of $X_{\sigma}$.
 
Now, we provide an algorithm allowing us to construct a K-R partition of  
$X_\sigma$ using the towers from (\ref{exp}). Let $\{a_1, a_2, a_3,...
\}$ be an enumeration of $\A$. Define a total order  $<$ on 
$\A$, by $a_i< a_j$ if and only if $i<j$ for $i,j \in \N$. The order $<$ generates
 the lexicographic order (denoted by $\ll$) on elements 
$a_ia_j \in \mathcal{L}_2(\sigma)$. This is a total order on the set $
\mathcal{L}_2(\sigma)$. Arrange the towers in expression (\ref{exp}) in the 
increasing lexicographic order of the base elements from left to right, i.e., the tower with base $\sigma(a_ia_j)$ is on the left of the tower with the base 
$\sigma(a_ka_\ell)$ if 
and only of $a_ia_j \ll a_ka_\ell$. Let $\sigma(a_1a_2)$ be the base of the first 
tower (or leftmost tower, which we denote by $T_1$ for simplicity) in the arrangement of towers. Let a point $x \in 
X_{\sigma}$ lies in $\underset{k = a_1}{\overset{h_{a_1} + h_{a_2}  - 1}
{\bigsqcup}} \, T^k [\sigma (a_1a_2)]$, i.e. $x$ lies in upper levels of $T_1$. Proceed in lexicographic order from left to right and find the smallest tower (in the lexicographic order) to the right of the $T_1$ that also contains $x$. Remove this tower from the collection. Now repeat the process for the next (in the lexicographic order) remaining tower in the arrangement and so forth. Let $\mathcal{P}_2 \subset \mathcal{L}_2$ denote the collection of words $w$ of length two such that $[\sigma(w)]$ form the base of towers in the final collection. This gives us (\ref{KR2}) for $n =2$. The proof for $n>2$ follows similarly.
\end{proof}

Proposition \ref{base 1} and Corollary \ref{CS} are alternate versions of Proposition \ref{base 2}
and Corollary \ref{CS 2}. Here we do not require the substitution to be of bounded size. 

\begin{proposition}\label{base 1} Let $\sigma$ be a left determined substitution on a countable alphabet $\A$. We define a sequence of Borel sets $\{B_n\}_{\N_0}$ as follows: Put $X_{\sigma} = B_0 $ and for $n \in \N$, 
\begin{equation} \label{com sec}
    B_n = \underset{a_{1}\cdots a_{n} \in \mathcal{P}_n(\sigma)}{\bigsqcup} [\sigma (a_{1} \cdots a_{n})]
\end{equation} where $\mathcal{P}_n \subset \mathcal{L}_n$ is a set of words of length $n$ as in Theorem \ref{RK_n}.  Then the set $\underset{n\in \N_0}{\bigcap} B_n $ is at most countably infinite. 

\end{proposition}

\begin{proof} The proof of this proposition is similar to the proof of Proposition \ref{base 2}, hence we only provide a sketch. The fact that the 
cardinality of $\underset{n\in \N_0}{\bigcap} B_n $ is countably infinite is proved by showing that any non-empty set of the form $\underset{n \in \N}{\bigcap} [\sigma(a_1\cdots a_n)]$, where $a_{1} \cdots a_{n} \in \mathcal{L}_n(\sigma)$ for every $n \in \N$, is a singleton set and that there exist countably many infinite non-empty sets of such form. This statement follows from the fact that $\sigma$ is left determined by applying arguments similar to that used in the proof of Proposition \ref{base 2}.
 \end{proof}

\begin{corollary}\label{CS} $\mathrm{(Wandering)}$ Let $\sigma $ be a left determined substitution on a countably infinite alphabet $\A$ and $(X_{\sigma},T)$ be the corresponding subshift. Then the sequence of Borel sets $(B_n)_{n\in \N_0} \subset X_{\sigma}$, defined in (\ref{com sec}), has following properties,

\begin{enumerate}[(a)]
\item $X_{\sigma} = B_0 \supset B_1 \supset B_2 \supset B_3 \cdots\,\,\,.$
\item The cardinality of $B_{\infty}: =\underset{n\in \N_0}{\bigcap} B_n$ is countably infinite and it is a wandering set.
\item $B_n$ is a complete $T$-section for each $n \in \N_0$. 
\item For each $n \in \N_0$ every point in $B_n$ is recurrent.
\end{enumerate}
\end{corollary}

\begin{proof} The proof is similar to Corollary \ref{CS 2}, hence we omit it. 
\end{proof}

\subsection{Non-stationary generalized Bratteli-Vershik model for substitutions on infinite alphabet}\label{BBD2} In Section \ref{Sec sub-BD} we constructed nested sequence of complete sections for a subshift associated with substitution on a countably infinite alphabet (see Corollary \ref{CS 2}). We used Corollary \ref{CS 2}, in Subsection \ref{BBD1} to construct a \textit{stationary} generalized Bratteli-Vershik model for a bounded size left determined substitution $\sigma$ on a countably infinite alphabet. 

Note that Corollary \ref{CS} provided an alternate construction of a nested sequence of complete sections (the substitution is not required to be of bounded size).  If we repeat the construction in Subsection \ref{BBD1} using Corollary \ref{CS} (instead of  Corollary \ref{CS 2}), we will obtain 
alternative generalized Bratteli-Vershik models for a left determined substitution on a countably infinite alphabet. Since the complete sections in Corollary \ref{CS} are not powers of $\sigma$, an ordered generalized Bratteli diagram constructed this way will not be stationary in general. Nevertheless, by the same reasoning as in the proof of Theorem \ref{isom1},  the dynamical system on the path space of this ordered (not necessarily stationary) generalized Bratteli diagram will still be isomorphic to $(X_\sigma, T)$. We want to emphasize that in Corollary \ref{CS} (unlike Corollary \ref{CS 2}) there is no requirement for the substitution to be of bounded size. Hence this construction will allow us to build Bratteli-Vershik models for any left determined substitution. We obtain the following theorem :

\begin{theorem}\label{isom2} Let $\sigma $ be a left determined substitution on a countably infinite alphabet and $(X_\sigma, T)$ be the corresponding subshift. Then there exists an ordered (not necessarily stationary) generalized Bratteli diagram $B = (V, E, \geq)$ and a Vershik map $\varphi: Y_B \rightarrow Y_B$ such that $(X_\sigma, T)$ is isomorphic to $(Y_B, \varphi)$.

\end{theorem}

\subsection{Proof of Theorem \ref{inv1}}\label{proof} In this subsection, we provide proof of Theorem \ref{inv1} for completion. We recall the statement below. 

\begin{theorem} [Theorem 2.20, \cite{Bezuglyi_Jorgensen_2021}]  Let $B = B(F)$ be a stationary generalized Bratteli diagram such that the incidence matrix $F$ is irreducible, aperiodic and recurrent. Then

\noindent $(1)$ there exists a tail invariant measure $\mu$ on the path space $Y_B$,

\noindent $(2)$  the measure $\mu$ is finite if and only if the left 
Perron-Frobenius eigenvector $\ell = (\ell_v)$ has the property $\sum_{v} \ell_v < \infty$.
\end{theorem}

\begin{proof} We identify the sets $V_i$ with a countably infinite set 
$V$.  Since $F$ is irreducible, aperiodic, and recurrent, there exists a 
Perron-Frobenius eigenvalue  $\lambda$ for $F$ (see Theorem \ref{PF1}). 
We denote by $\ell = (\ell_v)$ a left eigenvector corresponding to $\lambda$ indexed by elements of $V$. Let $\ol e(w, v)$ denote a finite path that begins at $w \in V_0$ and ends at $v \in V_n$, $n \in \N$. To  define the measure 
$\mu$, we find its values on all cylinder sets, and then we check that this 
definition can be extended to a Borel measure using the Kolmogorov 
consistency theorem. For the cylinder set  $[\ol e(w, v)]$, we set
 \be\label{eq inv meas left}
 \mu([\ol e(w, v)]) = \frac{\ell_v}{\lambda^{n}}. 
 \ee
 To see that $\mu$ can be extended
to a Borel measure,  let  $\ol g_u = \ol e(w, v)e(v,u)$ denote the concatenation of path $\ol e(w, v)$ with an edge  $e(v, u)$ where  $u \in V_{n+1}\cap 
r(s^{-1}(v))$.  Next, we compute the measure of the set  
\begin{equation*}
\bigcup_{u \in V_{n+1} \cap r(s^{-1}(v))} [\ol g_u]
\end{equation*}
and show that it is equal to the measure of the set $[e(w, v)]$.
For this,  we use the relation 
$\ell F = \lambda \ell$:
$$
\mu \Big( \bigcup_{u \in V_{n+1} \cap r(s^{-1}(v))} [\ol g_u]\Big) 
 = \sum_{u \in V_{n+1}} f_{uv}\frac{l_u}{\lambda^{n+1}} = 
\frac{\ell_v}{\lambda^{n}} = \mu([\ol e(w, v)]).
$$ Measure $\mu$ is tail invariant since any two cylinder sets defined by finite paths terminating at the same vertex $v \in V_n$ (say (say $\ol e(w_1, v)$ and $\ol e(w_2, v)$) ) have the same measure as given by 
$$
\mu([\ol e(w_1, v)]) = \frac{\ell_v}{\lambda^{n}} = \mu([\ol e(w_2, v)]).
$$
 This proves part $(1)$ of the theorem.

To see that $(2)$ holds, we observe  that the set  $Y_B$ is the 
union of all subsets $Y_B(w)$ where $Y_B(w) = \{ x = (x_i) \in Y_B : s(x_0) =
 w\}$. Then 
 $$
 \mu(Y_B)  = \sum_{w \in V_0} \mu(Y_B(w) ) = \sum_{v \in V} \ell_v.
 $$
The measure $\mu$ is finite if and only if $\sum_{v \in V} \ell_v < \infty$.
 \end{proof}

\bibliographystyle{alpha}
\bibliography{references1.bib}

\def\ocirc#1{\ifmmode\setbox0=\hbox{$#1$}\dimen0=\ht0 \advance\dimen0 by1pt\rlap{\hbox to\wd0{\hss\raise\dimen0 \hbox{\hskip.2em$\scriptscriptstyle\circ$}\hss}}#1\else {\accent"17 #1}\fi}
\begin{thebibliography}{BKMS10}

\bibitem[BBEP21]{Beal_Berstel_Eilers_2021}
Marie-Pierre B\'{e}al, Jean Berstel, S\o~ren Eilers, and Dominique Perrin.
\newblock Symbolic dynamics.
\newblock In {\em Handbook of automata theory. {V}ol. {II}. {A}utomata in mathematics and selected applications}, pages 987--1030. EMS Press, Berlin, [2021] \copyright 2021.

\bibitem[BDK06]{BezuglyiDooleyKwiatkowski_2006}
S.~Bezuglyi, A.~H. Dooley, and J.~Kwiatkowski.
\newblock Topologies on the group of {B}orel automorphisms of a standard {B}orel space.
\newblock {\em Topol. Methods Nonlinear Anal.}, 27(2):333--385, 2006.

\bibitem[BJ21]{Bezuglyi_Jorgensen_2021}
Sergey Bezuglyi and Palle E.~T. Jorgensen.
\newblock Harmonic analysis invariants for infinite graphs via operators and algorithms.
\newblock {\em J. Fourier Anal. Appl.}, 27(2):Paper No. 34, 46, 2021.

\bibitem[BK96]{BeckerKechris_1996}
H.~Becker and A.~S. Kechris.
\newblock {\em The descriptive set theory of {P}olish group actions}, volume 232 of {\em London Mathematical Society Lecture Note Series}.
\newblock Cambridge University Press, Cambridge, 1996.

\bibitem[BKM09]{Bezuglyi_Kwiatkowski_Medynets_2009}
S.~Bezuglyi, J.~Kwiatkowski, and K.~Medynets.
\newblock Aperiodic substitution systems and their {B}ratteli diagrams.
\newblock {\em Ergodic Theory Dynam. Systems}, 29(1):37--72, 2009.

\bibitem[BKMS10]{B_K_M_S_2010}
S.~Bezuglyi, J.~Kwiatkowski, K.~Medynets, and B.~Solomyak.
\newblock Invariant measures on stationary {B}ratteli diagrams.
\newblock {\em Ergodic Theory Dynam. Systems}, 30(4):973--1007, 2010.

\bibitem[BM16]{Banerjee_McPhee_2016}
Joydeep Banerjee and John McPhee.
\newblock Graph-theoretic sensitivity analysis of multi-domain dynamic systems: theory and symbolic computer implementation.
\newblock {\em Nonlinear Dynam.}, 85(1):203--227, 2016.

\bibitem[Bra72]{Bratteli_1972}
Ola Bratteli.
\newblock Inductive limits of finite dimensional {$C^{\ast} $}-algebras.
\newblock {\em Trans. Amer. Math. Soc.}, 171:195--234, 1972.

\bibitem[BSTY19]{Berthe_Steiner_Thuswaldner_Yassawi_2019}
Val\'{e}rie Berth\'{e}, Wolfgang Steiner, J\"{o}rg~M. Thuswaldner, and Reem Yassawi.
\newblock Recognizability for sequences of morphisms.
\newblock {\em Ergodic Theory Dynam. Systems}, 39(11):2896--2931, 2019.

\bibitem[Cas94]{Cassaigne_1994}
Julien Cassaigne.
\newblock An algorithm to test if a given circular {HD}0{L}-language avoids a pattern.
\newblock In {\em Information processing '94, {V}ol. {I} ({H}amburg, 1994)}, IFIP Trans. A Comput. Sci. Tech., A-51, pages 459--464. North-Holland, Amsterdam, 1994.

\bibitem[DHS99]{Durand_Host_Skau_1999}
F.~Durand, B.~Host, and C.~Skau.
\newblock Substitutional dynamical systems, {B}ratteli diagrams and dimension groups.
\newblock {\em Ergodic Theory Dynam. Systems}, 19(4):953--993, 1999.

\bibitem[DJK94]{DoughertyJacksonKechris_1994}
R.~Dougherty, S.~Jackson, and A.~S. Kechris.
\newblock The structure of hyperfinite {B}orel equivalence relations.
\newblock {\em Trans. Amer. Math. Soc.}, 341(1):193--225, 1994.

\bibitem[DK19]{DownarowiczKarpel_2019}
Tomasz Downarowicz and Olena Karpel.
\newblock Decisive {B}ratteli-{V}ershik models.
\newblock {\em Studia Math.}, 247(3):251--271, 2019.

\bibitem[Dom12]{Dombek2012}
Daniel Dombek.
\newblock Substitutions over infinite alphabet generating {$(-\beta)$}-integers.
\newblock {\em Internat. J. Found. Comput. Sci.}, 23(8):1627--1639, 2012.

\bibitem[DOP18]{Durand_Ormes_Petite_2018}
Fabien Durand, Nicholas Ormes, and Samuel Petite.
\newblock Self-induced systems.
\newblock {\em J. Anal. Math.}, 135(2):725--756, 2018.

\bibitem[Dur10]{Durand_2010}
Fabien Durand.
\newblock Combinatorics on {B}ratteli diagrams and dynamical systems.
\newblock In {\em Combinatorics, automata and number theory}, volume 135 of {\em Encyclopedia Math. Appl.}, pages 324--372. Cambridge Univ. Press, Cambridge, 2010.

\bibitem[Fer06]{Ferenczi_2006}
S\'{e}bastien Ferenczi.
\newblock Substitution dynamical systems on infinite alphabets.
\newblock {\em Ann. Inst. Fourier (Grenoble)}, 56:2315--2343, 2006.

\bibitem[FM77]{FeldmanMoore_1977}
Jacob Feldman and Calvin~C. Moore.
\newblock Ergodic equivalence relations, cohomology, and von {N}eumann algebras. {I}.
\newblock {\em Trans. Amer. Math. Soc.}, 234(2):289--324, 1977.

\bibitem[Fog02]{Fogg2002}
N.~Pytheas Fogg.
\newblock {\em Substitutions in dynamics, arithmetics and combinatorics}, volume 1794 of {\em Lecture Notes in Mathematics}.
\newblock Springer-Verlag, Berlin, 2002.

\bibitem[For97]{Forrest_1997}
A.~H. Forrest.
\newblock {$K$}-groups associated with substitution minimal systems.
\newblock {\em Israel J. Math.}, 98:101--139, 1997.

\bibitem[GPS95]{Giordano_Putnam_Skau_1995}
Thierry Giordano, Ian~F. Putnam, and Christian~F. Skau.
\newblock Topological orbit equivalence and {$C^*$}-crossed products.
\newblock {\em J. Reine Angew. Math.}, 469:51--111, 1995.

\bibitem[Gra11]{Gray2011}
Robert~M. Gray.
\newblock {\em Entropy and information theory}.
\newblock Springer, New York, second edition, 2011.

\bibitem[GS98]{GurevichShevchenko1998}
B.~M. Gurevich and S.~V. Savchenko.
\newblock Thermodynamic formalism for symbolic {M}arkov chains with a countable number of states.
\newblock {\em Uspekhi Mat. Nauk}, 53(2(320)):3--106, 1998.

\bibitem[GW95]{GlasnerWeiss_1995}
Eli Glasner and Benjamin Weiss.
\newblock Weak orbit equivalence of {C}antor minimal systems.
\newblock {\em Internat. J. Math.}, 6(4):559--579, 1995.

\bibitem[Hjo00]{Hjorth_2000}
Greg Hjorth.
\newblock {\em Classification and orbit equivalence relations}, volume~75 of {\em Mathematical Surveys and Monographs}.
\newblock American Mathematical Society, Providence, RI, 2000.

\bibitem[HPS92]{HermanPutnamSkau_1992}
Richard~H. Herman, Ian~F. Putnam, and Christian~F. Skau.
\newblock Ordered {B}ratteli diagrams, dimension groups and topological dynamics.
\newblock {\em Internat. J. Math.}, 3(6):827--864, 1992.

\bibitem[Jea16]{Jeandel_2016}
Emmanuel Jeandel.
\newblock Computability in symbolic dynamics.
\newblock In {\em Pursuit of the universal}, volume 9709 of {\em Lecture Notes in Comput. Sci.}, pages 124--131. Springer, [Cham], 2016.

\bibitem[JKL02]{JacksonKechrisLouveau_2002}
S.~Jackson, A.~S. Kechris, and A.~Louveau.
\newblock Countable {B}orel equivalence relations.
\newblock {\em J. Math. Log.}, 2(1):1--80, 2002.

\bibitem[JKL14]{Jaerisch2014}
Johannes Jaerisch, Marc Kesseb\"{o}hmer, and Sanaz Lamei.
\newblock Induced topological pressure for countable state {M}arkov shifts.
\newblock {\em Stoch. Dyn.}, 14(2):1350016, 31, 2014.

\bibitem[Kec95]{Kechris_1995}
Alexander~S. Kechris.
\newblock {\em Classical descriptive set theory}, volume 156 of {\em Graduate Texts in Mathematics}.
\newblock Springer-Verlag, New York, 1995.

\bibitem[Kec19]{Kechris_2019}
Alexander~S. Kechris.
\newblock The theory of countable borel equivalence relations.
\newblock {\em preprint}, 2019.

\bibitem[Kit98]{Kitchens1998}
Bruce~P. Kitchens.
\newblock {\em Symbolic dynamics}.
\newblock Universitext. Springer-Verlag, Berlin, 1998.
\newblock One-sided, two-sided and countable state Markov shifts.

\bibitem[KM04]{KechrisMiller_2004}
Alexander~S. Kechris and Benjamin~D. Miller.
\newblock {\em Topics in orbit equivalence}, volume 1852 of {\em Lecture Notes in Mathematics}.
\newblock Springer-Verlag, Berlin, 2004.

\bibitem[KS19]{Klouda_Starosta_2019}
Karel Klouda and \v{S}t\v{e}p\'{a}n Starosta.
\newblock Characterization of circular {D}0{L}-systems.
\newblock {\em Theoret. Comput. Sci.}, 790:131--137, 2019.

\bibitem[Mar73]{Martin_1973}
John~C. Martin.
\newblock Minimal flows arising from substitutions of non-constant length.
\newblock {\em Math. Systems Theory}, 7:72--82, 1973.

\bibitem[Mau06]{Mauduit2006}
Christian Mauduit.
\newblock Propri\'et\'es arithm\'etiques des substitutions et automates infinis.
\newblock {\em Annales de l'Institut Fourier}, 56(7):2525--2549, 2006.

\bibitem[Med06]{Medynets_2006}
Konstantin Medynets.
\newblock Cantor aperiodic systems and {B}ratteli diagrams.
\newblock {\em C. R. Math. Acad. Sci. Paris}, 342(1):43--46, 2006.

\bibitem[Mos92]{Mosse_1992}
Brigitte Moss\'{e}.
\newblock Puissances de mots et reconnaissabilit\'{e} des points fixes d'une substitution.
\newblock {\em Theoret. Comput. Sci.}, 99(2):327--334, 1992.

\bibitem[Mos96]{Mosse_1996}
Brigitte Moss\'{e}.
\newblock Reconnaissabilit\'{e} des substitutions et complexit\'{e} des suites automatiques.
\newblock {\em Bull. Soc. Math. France}, 124(2):329--346, 1996.

\bibitem[MRW21]{Manibo_Rust_Walton_2021}
Neil Mañibo, Dan Rust, and James~J. Walton.
\newblock Spectral properties of substitutions on compact alphabets, 2021.

\bibitem[MS93]{Mignosi_Seebold_1993}
Filippo Mignosi and Patrice S\'{e}\'{e}bold.
\newblock If a {DOL} language is {$k$}-power free then it is circular.
\newblock In {\em Automata, languages and programming ({L}und, 1993)}, volume 700 of {\em Lecture Notes in Comput. Sci.}, pages 507--518. Springer, Berlin, 1993.

\bibitem[Nad13]{Nadkarni_2013}
M.~G. Nadkarni.
\newblock {\em Basic ergodic theory}, volume~6 of {\em Texts and Readings in Mathematics}.
\newblock Hindustan Book Agency, New Delhi, third edition, 2013.

\bibitem[OTW14]{OttTomfordeWillis2014}
William Ott, Mark Tomforde, and Paulette~N. Willis.
\newblock {\em One-sided shift spaces over infinite alphabets}, volume~5 of {\em New York Journal of Mathematics. NYJM Monographs}.
\newblock State University of New York, University at Albany, Albany, NY, 2014.

\bibitem[PFS14]{Frank_Sadun_2014}
Natalie Priebe~Frank and Lorenzo Sadun.
\newblock Fusion tilings with infinite local complexity.
\newblock {\em Topology Proc.}, 43:235--276, 2014.

\bibitem[Que10]{Queffelec_2010}
Martine Queff\'{e}lec.
\newblock {\em Substitution dynamical systems---spectral analysis}, volume 1294 of {\em Lecture Notes in Mathematics}.
\newblock Springer-Verlag, Berlin, second edition, 2010.

\bibitem[RY17]{RowlandYassawi2017}
Eric Rowland and Reem Yassawi.
\newblock Profinite automata.
\newblock {\em Adv. in Appl. Math.}, 85:60--83, 2017.

\bibitem[Shi20]{Shimomura_2020}
Takashi Shimomura.
\newblock Bratteli-{V}ershik models and graph covering models.
\newblock {\em Adv. Math.}, 367:107127, 54, 2020.

\bibitem[Sol98]{Solomyak_1998}
B.~Solomyak.
\newblock Nonperiodicity implies unique composition for self-similar translationally finite tilings.
\newblock {\em Discrete Comput. Geom.}, 20(2):265--279, 1998.

\bibitem[Tak20]{Takahasi2020}
Hiroki Takahasi.
\newblock Entropy-approachability for transitive {M}arkov shifts over infinite alphabet.
\newblock {\em Proc. Amer. Math. Soc.}, 148(9):3847--3857, 2020.

\bibitem[Var63]{Varadarajan_1963}
V.~S. Varadarajan.
\newblock Groups of automorphisms of {B}orel spaces.
\newblock {\em Trans. Amer. Math. Soc.}, 109:191--220, 1963.

\bibitem[Ver81]{Vershik_1981}
A.~M. Vershik.
\newblock Uniform algebraic approximation of shift and multiplication operators.
\newblock {\em Dokl. Akad. Nauk SSSR}, 259(3):526--529, 1981.

\bibitem[Wei84]{Weiss_1984}
Benjamin Weiss.
\newblock Measurable dynamics.
\newblock In {\em Conference in modern analysis and probability ({N}ew {H}aven, {C}onn., 1982)}, volume~26 of {\em Contemp. Math.}, pages 395--421. Amer. Math. Soc., Providence, RI, 1984.

\end{thebibliography}
\end{document}